\documentclass[a4paper,11pt]{article}

\usepackage{indentfirst} 
\usepackage{amsmath, amsthm, amssymb,xfrac, mathrsfs, bbm} 
\usepackage{tikz} 
\usetikzlibrary{hobby}
\usetikzlibrary{decorations.pathreplacing}
\usepackage{wrapfig} 
\usepackage{xcolor} 
\usepackage{verbatim} 
\usepackage{enumerate} 
\usepackage{fullpage}
\usepackage{todonotes}

\usepackage{hyperref} 
\hypersetup{
colorlinks=true,
citecolor=black!50!red,
linkcolor=black!50!red,
linktoc=all
}
\usepackage[all]{hypcap} 

\numberwithin{equation}{section}
\newtheorem{theorem}{Theorem}[section]
\newtheorem{definition}[theorem]{Definition}
\newtheorem{proposition}[theorem]{Proposition}
\newtheorem{lemma}[theorem]{Lemma}
\newtheorem{corollary}[theorem]{Corollary}

\theoremstyle{definition}
\newtheorem{remark}[theorem]{Remark}

\newcounter{constantk}
\newcommand{\newconstantk}[1]{\refstepcounter{constantk}\label{#1}}
\newcommand{\useconstantk}[1]{k_{\textnormal{\tiny \ref{#1}}}}
\newcounter{constant}
\newcommand{\newconstant}[1]{\refstepcounter{constant}\label{#1}}
\newcommand{\useconstant}[1]{c_{\textnormal{\tiny \ref{#1}}}}
\newcounter{bigconstant}
\newcommand{\newbigconstant}[1]{\refstepcounter{bigconstant}\label{#1}}
\newcommand{\usebigconstant}[1]{C_{\textnormal{\tiny \ref{#1}}}}
\DeclareMathOperator\PP{P}
\DeclareMathOperator\EE{E}
\DeclareMathOperator\dist{d}
\newcommand\dd{\textnormal{d}}

\newcommand{\1}{{\mathbbm 1}}

\newcommand{\LL}{\mathbb{L}}
\newcommand{\Z}{\mathbb{Z}}
\newcommand{\R}{\mathbb{R}}
\newcommand{\N}{\mathbb{N}}

\newcommand{\cC}{\mathcal{C}}
\newcommand{\cF}{\mathcal{F}}

\newcommand{\cI}{\mathcal{I}}
\newcommand{\cO}{\mathcal{O}}
\newcommand{\cP}{\mathcal{P}}

\makeindex

\begin{document}

\title{Law of large numbers for ballistic random walks in dynamic random environments under lateral decoupling}

\date{\today}
\author{
Weberson~S.~Arcanjo
  \thanks{Email: \ \texttt{weberson.arcanjo@gmail.com}; \ Department of Mathematics, Universidade Federal de Minas Gerais, Av.\ Antonio Carlos 6627, 31270-901 Belo Horizonte, MG - Brazil.}
  \and
  Rangel Baldasso
  \thanks{Email: \ \texttt{r.baldasso@math.leidenuniv.nl}; \ Mathematical Institute, Leiden University, P.O. Box 9512, 2300 RA Leiden, The Netherlands.}
  \and
  Marcelo~R.~Hil\'ario
  \thanks{Email: \ \texttt{mhilario@mat.ufmg.br}; \ Department of Mathematics, Universidade Federal de Minas Gerais, Av.\ Antonio Carlos 6627, 31270-901 Belo Horizonte, MG - Brazil.}
  \and
  Renato~S.~dos~Santos
  \thanks{Email: \ \texttt{rensosan@gmail.com}; \ Department of Mathematics, Universidade Federal de Minas Gerais, Av.\ Antonio Carlos 6627, 31270-901 Belo Horizonte, MG - Brazil.}
}

\maketitle

\begin{abstract}
We establish a strong law of large numbers for one-dimensional continuous-time random walks in dynamic random environments under two main assumptions: 
the environment is required to satisfy a decoupling inequality that can be interpreted as a bound on the speed of dependence propagation, 
while the random walk is assumed to move ballistically with a speed larger than this bound. 
Applications include environments with strong space-time correlations such as the zero-range process and the asymmetric exclusion process. 

\medskip

  \noindent
  \emph{Keywords and phrases.}
  Random walks, dynamical random environment, zero-range process, asymmetric exclusion process.

  \noindent
  MSC 2010: 60K35, 82B43.
\end{abstract}


\section{Introduction}

Random walks in random environments model the motion of a particle subjected to the influence of an inhomogeneous medium. This is done by specifying a transition kernel that depends locally on an underlying collection of random variables indexed by space, called \emph{random environment}.
The environment may be \emph{static} or \emph{dynamic}, according to whether it remains fixed or evolves stochastically in time. 

In this work, we consider one-dimensional dynamic random environments that are invariant under space-time shifts and satisfy a mild decorrelation inequality (cf.\ Assumption~\eqref{e:DEC}) that can be interpreted as a bound on the speed of dependence propagation.
We then introduce a continuous-time random walk $X$ whose jump rates are bounded and depend locally on the environment.
Our main result, Theorem~\ref{t:maingeneral} below, states that, if the random walk moves ballistically with large enough speed (cf.\ Assumption~\eqref{e:BAL}), then it satisfies a strong law of large numbers.
As applications, we derive new strong laws of large numbers for random walks driven by the zero-range process (Corollary \ref{c:ZR}) and the asymmetric exclusion process (Corollary \ref{c:LLN_ASEP}) under the hypothesis that the random walk is subjected to sufficiently large local drifts uniformly over the environment.
We also mention that, although our methods apply in the static case, we do not believe that they bring any novelty in this setting.

In applications of Theorem~\ref{t:maingeneral}, once the random walk is shown to behave ballistically, one only needs to verify Assumption~\eqref{e:DEC} for the environment in order to obtain the law of large numbers. 
This assumption states roughly that events depending on disjoint regions of the environment become nearly independent as long as the spatial distance between these regions is large enough compared to their separation in time.
It is satisfied by a large class of environments including some conservative interacting particle systems which lack uniform mixing bounds.

Our setting is similar to the one in \cite{avena2011transient}, where a law of large numbers and a central limit theorem were proved for the random walk on the simple symmetric exclusion process under the assumption of sufficiently large local drifts. 
There, the relation with the interchange process is explored in order to construct a renewal structure. Their local drift assumption was relaxed in \cite{huveneers2015random}, where the renewal strategy was improved and combined with a renormalization scheme in order to study perturbative cases characterized by very large or very small rates of evolution of the environment. Environments composed of particles performing independent random walks were treated similarly in \cite{den2014random, hilario2015random, BHSST19, BHSST20} in other perturbative regimes.

In \cite{blondel2020random}, a strong law of large numbers was obtained for dynamic random environments with sufficiently fast correlation decay in time.
The work \cite{HKA19} built on the methods from~\cite{blondel2020random} together with the renewal structure from \cite{huveneers2015random} to establish limit theorems for the random walk in the symmetric exclusion process without imposing large drifts or perturbative rates on the evolution of the environment. 
It explores the fact that the environment fulfills decorrelation inequalities which is different from the one that we consider here.
We further explore the renormalization techniques from \cite{blondel2020random, HKA19} to tackle new kinds of environments for which renewal structures are, to the best of our knowledge, not yet known, such as the zero-range process and the asymmetric exclusion process. 
Our adaptation of their methods provides a general result that can be applied to a large class of environments as soon as the necessary ballisticity of the random walk is verified.

For further discussion about our results and their relation with the literature on the topic, see Section \ref{ss:discussion}.

\subsection*{Outline of the paper}
\label{ss:outline}

\noindent
The paper is organized as follows. The mathematical setup and main results are given in Section~\ref{s:setting}, including our main 
assumptions, two applications (to the zero-range and asymmetric exclusion processes) and a discussion of related works.
Section~\ref{s:proofoverview} provides a proof overview for our main theorem via intermediate results that are proved in Sections~\ref{s:Proof_decaiv} and~\ref{s:v-=v+}. The proof of the main theorem is then completed in Section~\ref{s:proof_main}. 
The remaining sections contain the proofs of some technical results used elsewhere in the paper: 
Section~\ref{s:asep} concerns a decoupling inequality for the simple exclusion process, and Appendix~\ref{s:martingale_appendix} deals with deviation estimates for submartigales.

\bigskip

\noindent \textbf{Acknowledgements.}
We thank Augusto Teixeira for proposing the problem and for fruitful discussions held during a visit to the Mathematics Department of UFMG.
We also thank Luca Avena, Tertuliano Franco and Bernardo de Lima for very useful comments and suggestions, and an anonymous referee for suggestions and corrections.
WA was supported by CAPES fellowship 88887.197372/2018-00 during the elaboration of this work.
RB has counted on the support of the Mathematical Institute of Leiden University.
The research of MH was partially supported by CNPq grant 312227/2020-5 (Produtividade em Pesquisa) and FAPEMIG grant APQ-01214-21 (Universal). 
The research of RSdS was partially supported by CNPq grant 313921/2020-2 (Produtividade em Pesquisa) and FAPEMIG grant APQ-02288-21 (Universal).

\section{Mathematical setting and main results}
\label{s:setting}

In this section we define the model and state our main results. 
In Section~\ref{ss:DRE} we describe the class of random environments that we will consider and state the decoupling assumption that they must satisfy. 
The random walk is constructed in Section~\ref{ss:RW}. 
The ballisticity assumption and our results in the general setting are given in Section~\ref{ss:resultsgeneral}, followed by our results in the case of the zero-range and asymmetric exclusion processes in Sections~\ref{ss:resultsZR} and~\ref{ss:resultsASEP}, respectively. 
In Section~\ref{ss:discussion} we give a brief historical overview of the model and contextualize our results within the literature. 
Section~\ref{ss:properties} collects some first basic properties of our setup. 

\medskip
\noindent \textbf{Notation.}  
We write $\mathbb{N}=\{1,2,\ldots\}$, $\mathbb{N}_{0}=\mathbb{N}\cup \{0\}$, and $\R_{+} = [0,\infty)$. 
For $x,y \in \R$, $x \vee y = \max\{x,y\}$, $x \wedge y = \min\{x,y\}$, $\lfloor x \rfloor$ is the largest integer not larger than $x$ and
 $\lceil x \rceil$ the smallest integer not smaller than $x$.
Given $y=(x,t) \in \R^2$, we write
\begin{equation*}
\pi_{1}(y)= x \quad \text{ and } \quad \pi_{2}(y)= t
\end{equation*} 
for its projections onto the space and time coordinates, respectively.
Finally, the set
\begin{equation}
\label{e:defLL}
\LL:=\Z\times \mathbb{R}_{+}
\end{equation}
will be useful as it contains the space-time position of our random walks.

\medskip
\noindent \textbf{Constants.}
We will denote by $C_1,C_2,..., c_1, c_2,...$ and $k_1, k_2,... $ positive constants whose values are fixed at their first appearance in the text. 
Dependence on other parameters may be indicated at the first appearance and omitted in future appearances. 
For example, if $c_i$ depends on $\rho$, we may write $c_i(\rho)$ at the first appearance and only write $c_i$ afterwards.

\subsection{The dynamic random environment}
\label{ss:DRE}

Let $(\Omega, \mathcal{F}, \PP)$ be a probability space and $(E, \mathcal{E})$ a measurable space. 
The role of random environment will be taken by a stochastic process $\eta = (\eta_t)_{t \in \R_+}$ on $(\Omega, \mathcal{F}, \PP)$
where $\eta_t = (\eta_t(x))_{x \in \Z} \in E^\Z$ for each $t \in \R_+$. 
We will assume the following:
\begin{enumerate}
\item[(E1)] the map $\eta:\R_+\times \Omega \to E^\Z$ is measurable;
\item[(E2)] for each $(x,t) \in \LL = \Z \times \R_+$, the space-time translation $\theta_{x,t} \eta$ defined by $(\theta_{x,t} \eta)_s(y) = \eta_{t+s}(x+y)$ has, under $\PP$, the same distribution as $\eta$.
\end{enumerate}

In order to state our decoupling assumption, we first need a few definitions.
For a fixed region $R \subset \R^2$, we say that a random variable on $(\Omega, \mathcal{F}, \PP)$ is \emph{supported on $R$}
if it is measurable with respect to $\sigma\big( \eta_t(x) \colon\, (x,t) \in R \cap \LL \big)$.

Given two regions $B_1,B_2 \subset \mathbb{R}^2$, we define their vertical distance as
\begin{equation*}
\dist_V = \dist_V(B_1, B_2)=\inf\{ |t-s| : \exists\, x,y \in \R \text{ such that } (x,t) \in B_1 \text{ and }(y,s)\in B_2\},
\end{equation*}
and their horizontal distance as
\begin{equation*}
\dist_H = \dist_H(B_1, B_2) =\inf \{|x-y| : \exists\, t,s \in \R \text{ such that } (x,t) \in B_1 \text{ and } (y,s) \in B_2\}.
\end{equation*}

We write $\EE$ for expectation with respect to $\PP$.
Our key assumption on $\eta$ is as follows.
It is inspired by \cite{HKA19}, see in particular 
Proposition~4.1 therein.

\newbigconstant{c:decoupling_2}
\newbigconstant{c:decoupling_3}

\medskip
\noindent {\bf Decoupling Assumption:}
There exist constants $v_\circ, \kappa_\circ, C_\circ, \usebigconstant{c:decoupling_2}, \usebigconstant{c:decoupling_3}>0$ and $\gamma_\circ>1$ such that, for any two regions of the form $B_1=(-\infty,a] \times [b, b+s]$, $B_2=[c,\infty) \times [d, d+s]$ with $a,c \in \R$ and $b,d,s \geq 0$ whose horizontal and vertical distances satisfy
\begin{equation}\label{condition_to_decouple}
\dist_H \geq  v_\circ \dist_V + \usebigconstant{c:decoupling_2} s + \usebigconstant{c:decoupling_3}
\end{equation}
and any pair of non-negative random variables $f_1$, $f_2$ supported respectively on $B_1$, $B_2$ and satisfying $||f_1||_{\infty}$, $||f_2||_{\infty}\leq 1$,
\begin{equation}
\label{e:DEC}
\EE[f_1f_2]\leq \EE[f_1] \EE[f_2]+C_\circ e^{-\kappa_\circ (\log\dist_H)^{\gamma_\circ}}.
\end{equation}

In applications, the height $s$ of relevant boxes will be small compared to their vertical separation $\dist_V$,
and thus the most important parameter in \eqref{condition_to_decouple} is $v_\circ$. This constant can be interpreted as
a bound on the speed of dependence propagation in the environment in the sense that, if $B_2$ does not intersect a cone with inclination slightly larger than $v_\circ$ stemming from $B_1$, then the state of the environment inside these boxes is decorrelated in the sense of \eqref{e:DEC}.

\subsection{The random walk}
\label{ss:RW}

We will consider random walks with nearest-neighbor jumps whose rates depend on the state of $\eta$ at the current position.
To this end, fix two measurable functions $\alpha, \beta:E \to \R_+$.
We will interpret $\alpha(\eta_t(x))$ as the rate for a random walk at site $x$ at time $t$ to jump to the right, 
and $\beta(\eta_t(x))$ as the rate to jump to the left.

We will assume that the total jump rate is bounded, i.e., for some $\Lambda \in [1,\infty)$,
\begin{equation}
\label{e:BR}
\sup_{\xi \in E} \left\{ \alpha(\xi) + \beta(\xi) \right\} \leq \Lambda.
\end{equation}

Recall $\LL = \Z \times \R_+$. To define our random walks, we enlarge the probability space $(\Omega, \mathcal{F}, \PP)$ to support an independent Poisson point process
$\Pi$ on $\LL \times [0,\Lambda]$ with intensity measure $\#(\dd z) \otimes \dd t \otimes \dd u$ where $\#(\dd z)$ denotes counting measure on $\Z$ and $\dd t$, $\dd u$ denote Lebesgue measure on $\R$, $[0,\Lambda]$, respectively. 
We will abuse notation and write $\Pi$ for both the point measure and its support.

Given a realization of $\eta$ and $\Pi$, define point measures $\Pi_\alpha$, $\Pi_\beta$ on $\LL$ by setting
\begin{align}\label{e:Pialpha}
\Pi_\alpha & = \left\{ (x,t) \in \LL \colon\, \exists\, u \in [0,\Lambda] \text{ such that } (x,t,u) \in \Pi \text{ and } u \leq \alpha(\eta_t(x)) \right\},
\end{align}
\begin{align}\label{e:Pibeta}
\Pi_\beta = \left\{ 
(x,t) \in \LL \colon\,
\begin{array}{ll}
\exists\, u \in [0,\Lambda] \text{ such that } (x,t,u) \in \Pi \text{ and} \\
\alpha(\eta_t(x)) < u \leq \alpha(\eta_t(x)) + \beta(\eta_t(x)) 
\end{array}
\right\}.
\end{align}
Note that $\Pi_\alpha$ and $\Pi_\beta$ above are projections of subsets of $\Pi$ onto $\mathbb{L}$, where we only keep points $(x,t) \in \mathbb{L}$ whose original point $(x,t,u) \in \Pi$ has suitable values of $u$.
By Poisson thinning, 
conditionally on $\eta$, the point measures $\Pi_\alpha$, $\Pi_\beta$ are independent Poisson point processes with intensity measures $\alpha(\eta_t(z)) \#(\dd z) \otimes \dd t$ and $\beta(\eta_t(z)) \#(\dd z) \otimes \dd t$, respectively.

For each $y = (x_0,t_0) \in \LL$, 
we define the random walk started at $y$ as 
the unique c\`adl\`ag path $X^y=(X^y_t)_{t \geq 0}$ such that $X^y_0 = \pi_1(y) = x_0$ and, for all $t>0$,
\begin{equation}
X^y_t =
\begin{cases}
 \text{$X^y_{t-} + 1$ \quad if $(X^y_{t-}, t_0+ t) \in \Pi_\alpha$};\\ 
 \text{$X^y_{t-} - 1$ \quad if $(X^y_{t-}, t_0 + t) \in \Pi_\beta$};\\ 
 \text{$X^y_{t-}$ \,\,\, \qquad otherwise}.
\end{cases}
\end{equation}

We write $Y^y = (Y^y_t)_{t \geq 0}$ for the space-time position of $X^y$, i.e., $Y^y_t =(X^y_t, t+t_0)$.

For simplicity we write $Y=Y^{(0,0)}$, $X=X^{(0, 0)}$, and we denote by $\mathfrak{X}$ the collection of random walks $(X^y)_{y \in \LL}$.
With this construction, conditionally on $\eta$, $X$ is a time-inhomogeneous Markov jump process that, when at site $x$ at time $t$, has jump rates 
$\alpha(\eta_{t}(x))$ to $x+1$ and $\beta(\eta_{t}(x))$ to $x-1$.
Note that, by translation invariance of $\eta$ and $\Pi$, $X^{(x,t)} - x$ has the same distribution as $X$ for each fixed $(x,t)$.

\subsection{Results: general setting}
\label{ss:resultsgeneral}

In order to state our law of large numbers, we introduce next our second main assumption,
regarding ballisticity of the random walk. Define $\log^+u = \log (u \vee 1)$.

\medskip
\noindent {\bf Ballisticity Assumption:}
there exist constants $v_\star, \kappa_\star, C_\star>0$ and $\gamma_\star>1$ such that
\begin{equation}
\label{e:BAL}
\PP \left( X_t \leq v_\star t \right) \leq C_\star e^{-\kappa_\star (\log^+ t)^{\gamma_\star}} \qquad \text{ for all } t  \geq 0.
\end{equation}

Our main theorem reads as follows.

\begin{theorem}
\label{t:maingeneral}
Suppose that Assumptions~\eqref{e:DEC} and \eqref{e:BAL} hold with $v_\star > v_\circ$, and let
\begin{equation}
\label{e:defgammakappa}
\gamma = \min\{\gamma_\circ, \gamma_\star\}, \qquad \kappa = \tfrac19 \min\{\kappa_\circ, \kappa_\star\}.
\end{equation}
There exists a speed $v \geq v_\star$ and, for any $\varepsilon>0$, there exists a constant $C_\varepsilon > 0$ such that
\begin{equation}
\label{e:LLNdecay}
\PP\left( \exists\, t \geq T \colon |X_t - vt| \geq \varepsilon t \right) \leq C_\varepsilon e^{-\kappa(\log^+ T)^{\gamma}} \quad \text{ for all } T \geq 0.
\end{equation}
In particular,
\begin{equation}\label{e:genLLN}
\lim_{t \to \infty} \frac{X_t}{t} = v \quad \PP\text{-almost surely and in } L^p \text{ for each } p \geq 1.
\end{equation}
\end{theorem}

\begin{remark}
Theorem~\ref{t:maingeneral} may be interpreted as follows. As previously mentioned, $v_{\circ}$ can be seen as an upper bound on the speed with which information spreads in the environment. 
If the random walk travels faster than this speed, then it eventually leaves behind all information that it uses, which introduces enough independence
for a law of large numbers to hold. 
\end{remark}

Assumption~\eqref{e:BAL} may be hard to verify in examples.
However, it is always true in the so-called \emph{non-nestling} regime, when all local drifts are sufficiently large. This leads us to the following result.

\begin{theorem}
\label{t:largedrift}
Suppose that Assumption~\eqref{e:DEC} holds and that
\[
\inf_{\xi \in E} \left\{ \alpha(\xi) - \beta(\xi) \right\} > v_\circ.
\]
Then Assumption~\eqref{e:BAL} holds with $v_\star > v_\circ$ and $\gamma_\star = \gamma_\circ$.
In particular, the conclusions of Theorem~\ref{t:maingeneral} are in force.
\end{theorem}

The proof of Theorem~\ref{t:largedrift} is given in Appendix~\ref{s:martingale_appendix}.

\subsection{Results: zero-range process}
\label{ss:resultsZR}

The one-dimensional zero-range process is an interacting particle system introduced by Spitzer in \cite{SPITZER70}. It is defined as a Markov process $\eta=(\eta_t)_{t\geq 0}$ with state-space $\N_0^\mathbb{Z}$, where $\eta_t(x)$ is interpreted as the number of particles occupying site $x$ at time $t$.
Starting from an initial configuration $\eta_0$, its evolution for $t \geq 0$ may be described as follows: 
independently for each $x$, with rate depending only on $\eta_t(x)$, a particle is chosen among those at $x$ and moved
to $x+1$ or $x-1$ with equal probability. With this dynamics, particles may interact only when they share the same site, hence the name \textit{zero-range}. 

Formally, fix a non-negative function $g : \mathbb{N}_{0} \to \mathbb{R}_+$ with $g(0) = 0$ and a translation invariant transition probability $p( \cdot, \cdot)$ on $\mathbb{Z}$. The zero-range process with rate function $g$ and transition probability $p$ is the continuous-time Markov process $\eta = (\eta_t)_{t\geq0}$ with state-space $\mathbb{N}_{0}^{\mathbb{Z}}$ whose infinitesimal generator $L$ acts on bounded local functions $f$ as
\begin{equation}\label{Gerador}
Lf(\eta)=\displaystyle\sum_{x \in \mathbb{Z}}g(\eta(x))\sum_{y \in \mathbb{Z}}p(x,y)[f(\eta^{xy})-f(\eta)],
\end{equation}
where $\eta^{xy}$ is obtained from $\eta$ by removing a particle from $x$ and placing it at $y$.
Here we will only consider the symmetric nearest-neighbor case, i.e.,  $p(0,\pm 1)= 1/2$. 

To guarantee the existence of the process, conditions on $g$ must be imposed.
Here we will assume that, for positive constants $\Gamma_- \leq \Gamma_+$,
\begin{equation}\label{ZR1}
\Gamma_- \leq g(k)-g(k-1)\leq \Gamma_+ \quad  \text{ for all } k \geq 1.
\end{equation}
This condition means that $g$ is increasing with a growth rate not far from linear, 
and implies that $\eta$ is well defined~\cite{ANDJEL82}.

We describe next the family of distributions that we will consider for the initial configuration $\eta_{0}$. 
For $\phi \geq 0$, consider the product measure with marginals $\hat{\mu}_\phi$ given by
\begin{equation}\label{distribuition}
\hat{\mu}_\phi(k)=\frac{1}{Z(\phi)}\frac{\phi^k}{g(k)!}, \qquad k \in \mathbb{N}_{0},
\end{equation}
where $g(k)!:=g(k)g(k-1)\cdot \cdot \cdot g(1)$, $g(0)!=1$ and $Z(\phi)$ is a normalizing constant.
This provides a one-parameter family of invariant distributions for the zero-range process \cite{ANDJEL82}.
When $g$ is linear, $\hat{\mu}_\phi$ is a Poisson distribution, corresponding to the case where particles perform independent random walks.

We call the mean number of particles per site $R(\phi) = \sum_{k} k \hat{\mu}_\phi (k)$ the \emph{density of particles}.
A straightforward computation reveals that $R(\phi)$ is an increasing bijection, so for $\rho \geq 0$ we may define $\mu_\rho$ as the product measure with marginals $\hat{\mu}_{R^{-1}(\rho)}$. We refer the reader to~\cite[Section 2.3]{KL98} for further details.

We will denote by $\PP^\rho$ the underlying probability measure on an enlarged probability space supporting both the zero-range process with initial distribution $\mu_\rho$ and the random walk $X$ constructed in Section~\ref{ss:RW}. We write $\EE^\rho$ for the corresponding expectation.

For the zero-range process as described above, the decoupling inequality \eqref{e:DEC} was obtained in \cite{BALDASSO17}, as stated next.
\begin{proposition}
\label{p:DEC_ZR}
Given two densities $0<\rho_- \leq \rho_+ < \infty$, there exist positive constants $v_\circ, \kappa_\circ, C_\circ, \usebigconstant{c:decoupling_2}, \usebigconstant{c:decoupling_3}$ such that, for any $\rho \in [\rho_-, \rho_+]$, Assumption~\eqref{e:DEC} holds with these constants and $\gamma_\circ = 5/4$ for the zero-range process under $\PP^\rho$.
\end{proposition}
\begin{proof} See \cite[Proposition~3.4.8]{BALDASSO17}.
The statement therein assumes that $B_1$, $B_2$ are square boxes with side-length $s$, but the same proof works under our assumptions.
See also \cite[Appendix~A]{Weberson}.
\end{proof}

Together with Theorems~\ref{t:maingeneral}--\ref{t:largedrift}, this implies the following.
\begin{corollary}
\label{c:ZR}
Given $0<\rho_- \leq \rho_+ < \infty$, there exists $v_\circ>0$ such that, if
\[
\displaystyle \inf_{\xi \in \mathbb{N}_0} \left\{\alpha(\xi)  - \beta(\xi) \right\} > v_\circ,
\]
then there exists $v:[\rho_-, \rho_+]\to (v_\circ, \infty)$ such that, for every $\rho \in [\rho_-, \rho_+]$,
\begin{equation}\label{e:LLN_ZR}
\lim_{t \to \infty} \frac{X_t}{t} = v(\rho) \qquad \PP^\rho\text{-almost surely.}
\end{equation}
\end{corollary}
\begin{proof}
Follows from Theorems~\ref{t:maingeneral}--\ref{t:largedrift} and Proposition~\ref{p:DEC_ZR}.
\end{proof}

\subsection{Results: asymmetric exclusion process}
\label{ss:resultsASEP}

The one-dimensional exclusion process is another interacting particle system introduced by Spitzer \cite{SPITZER70}, in which particles perform independent continuous-time random walks except for the \emph{exclusion rule}: if a particle tries to jump to a site that is already occupied, this jump is suppressed. It may be defined as a Feller process $\eta = (\eta_t)_{t \geq 0}$
with state-space $\{0,1\}^{\Z}$ whose generator $L$ acts on a bounded local function $f$ as
\begin{equation}
\label{e:geradorASEP}
Lf(\eta) = \sum_{x,y \in \Z} p(x,y) \left[ f(\eta^{xy}) - f(\eta)\right],
\end{equation}
where $p(\cdot,\cdot)$ is a translation-invariant transition probability and $\eta^{xy}$ is obtained from $\eta$ by interchanging the values of $\eta(x)$ and $\eta(y)$.
Here we will only consider the \emph{simple} case, i.e., when $p(0,1) = p = 1-p(0,-1)$, $p\in [0,1]$. 
The process is then called \emph{symmetric} when $p=\tfrac12$ and \emph{asymmetric} when $p \neq \tfrac12$.

As initial distribution, we fix $\rho \in [0,1]$ and take $\mu_\rho$ to be the product measure with marginals Bernoulli$(\rho)$.
This provides a family of invariant distributions for the exclusion process.
For more details, we refer the reader to \cite[Chapter~8]{Liggett}.

We denote by $\PP^{p,\rho}$ the underlying probability measure on an enlarged probability space supporting both the simple exclusion process with $p(0,1) = p$ and initial distribution $\mu_\rho$ and the random walk $X$ constructed in Section~\ref{ss:RW}.

A stronger form of the lateral decoupling inequality \eqref{e:DEC} was obtained for the simple symmetric exclusion process in \cite{HKA19}.
We provide next a version that holds also in the asymmetric case.
\begin{proposition}
\label{p:DEC_ASEP}
For any $v_\circ>1$, $\gamma_\circ>1$, and $\kappa_{\circ}, C_\circ > 0$, 
there exist positive constants $\usebigconstant{c:decoupling_2}, \usebigconstant{c:decoupling_3}$ 
such that, for any $p,\rho \in [0,1]$, Assumption~\eqref{e:DEC} holds with these constants
for the exclusion process under $\PP^{p,\rho}$.
\end{proposition}

Proposition~\ref{p:DEC_ASEP} will be proved in Section~\ref{s:asep}.
Together with Theorems~\ref{t:maingeneral} and~\ref{t:largedrift}, it implies the following law of large numbers.

\begin{corollary}
\label{c:LLN_ASEP}
Suppose that
$\inf_{\xi \in \{0,1\}} \left\{\alpha(\xi)  - \beta(\xi) \right\} > 1$.
Then there exists a function $v:[0,1]\times[0,1]\to (1,\infty)$ such that, for every $p,\rho \in [0,1]$,
\begin{equation}\label{e:LLN_ASEP}
\lim_{t \to \infty} \frac{X_t}{t} = v(p,\rho) \qquad \PP^{p,\rho}\text{-almost surely.}
\end{equation}
\end{corollary}
\begin{proof}
Follows from Theorems~\ref{t:maingeneral} and~\ref{t:largedrift}, and Proposition~\ref{p:DEC_ASEP}.
\end{proof}

\subsection{Related works and discussion of the results}
\label{ss:discussion}

We give next a brief and by no means exhaustive historical overview of the model, focussing on works more closely related to our setting and followed by some remarks to help contextualize our results within the literature.

Random walks in random environments have been the subject of intense research activity in probability.
Early studies on the topic were motivated by applications ranging from biophysics \cite{chernov67} to crystal growth \cite{temkin69}.
For static one-dimensional environments, refined mathematical results have been obtained including recurrence-transience criteria, laws of large numbers, central limit theorems and large deviation principles (see for instance \cite{solomon1975, kesten75, sinai82, alili99}).

In higher dimensions, the understanding is still much less detailed despite many advances over the last decades (see \cite{zeitouni2004random, sznitman2004topics, drewitz14} and references therein).
Much progress has been made for i.i.d.\ environments in the so-called \emph{ballistic regime}: namely, ballisticity conditions (e.g.\ the $(T)_{\gamma}$ conditions) have been identified that imply laws of large numbers, Brownian scaling and large deviation estimates (cf.\ \cite{Sznitman99, Sznitman01, Sznitman02}).
See also~\cite{berger13} for refinements leading to effective polynomial ballisticity conditions. 
Ballistic random walks in random environments satisfying strong mixing conditions were studied in \cite{cz} and more recently in \cite{guerra22}.

For dynamic random environments, initial progress was made mostly for environments that are independent in either space or time (cf.\ \cite{madras1986process, boldrighini04, dolgopyat_liverani}). 
The work \cite{conemixing} considers random environments given by interacting particle systems under a milder mixing condition known as \emph{cone-mixing},
which is adapted from a condition introduced in the static case in \cite{cz} and allows for 
an approximate renewal structure to be employed in order to prove a law of large numbers for the random walk. 
Roughly, the cone-mixing condition requires the random environment inside far away cones to become approximately independent of the initial configuration,
\emph{uniformly} over the realizations of the latter. 
This excludes some natural examples of interacting particle systems, notably conservative ones such as the exclusion process and the zero-range process, but also non-conservative ones such as the supercritical contact process.
Laws of large numbers, central limit theorems and large deviation estimates have been later obtained for non-uniformly mixing random environments, for example the contact process~\cite{den2014scaling, mountford2015random}, the exclusion process~\cite{avena2015symmetric, avena2011transient, dos2014non, HKA19, huveneers2015random}, and independent random walks \cite{den2014random, hilario2015random}.
Regeneration arguments play an important role in several of these works, especially in connection to the central limit theorem.

In \cite{hilario2015random}, a ballisticity condition reminiscent of the $(T)_\gamma$ conditions from~\cite{Sznitman02} was introduced 
in order to study an environment composed of independent random walks. Therein it is shown that, on the one hand, the ballisticity assumption together with
a regeneration argument imply the law of large numbers as well as the central limit theorem and deviation estimates.
On the other hand, refined decoupling properties of the environment together with a renormalization argument are used to show that the ballisticity
condition holds in a perturbative regime of high density, in particular implying that the previous limit theorems hold in this regime as well.
The works \cite{BHSST19, BHSST20} extended this investigation to higher dimensions and to the low density regime. 
In \cite{huveneers2015random}, similar arguments were developed for the case when the environment is given by the symmetric simple exclusion process; there the ballistic behavior is obtained in perturbative regimes of large or small jump rates for the environment.

Note that systems of independent random walks as considered in \cite{hilario2015random} are special instances of the zero-range process,
which is one of our main examples here. 
However, the regeneration strategy developed in \cite{hilario2015random} does not directly extend to cases with interaction. 
Similarly, the regeneration structure used in \cite{huveneers2015random} relies on the representation of the symmetric exclusion process in terms of the interchange process, which does not apply to asymmetric exclusion. 
More broadly, regeneration arguments tend to rely (as far as we are aware) on special characteristics of the model at hand and are thus difficult to generalize. One of our main motivations for this work was to obtain a law of large numbers in greater generality without using regeneration, 
which is achieved here by exploiting the multiscale renormalization methods from \cite{blondel2020random}.

\subsection{Basic properties of the construction}
\label{ss:properties}

We collect here some additional properties of our construction in Section~\ref{ss:RW} that will be useful in the sequel.
The first is a straightforward consequence:

\begin{enumerate}
\item[] (Coalescence). $\;$ If $X^{(x_0, t_0)}_t=x$ then $X^{(x_0, t_0)}_{t+s} = X^{(x,t+t_0)}_s$ for all $s \geq 0$.
\end{enumerate}

Recall the setup of Section~\ref{ss:RW}. 
Denote by $\Pi_\Lambda$ the Poisson point process on $\LL$ defined as $\Pi_\Lambda = \Pi_{\hat{\alpha}}$ (recall \eqref{e:Pialpha}), where 
$\hat{\alpha}:E \to \R_+$ is constant and equal to $\Lambda$, i.e., $\hat{\alpha}(\xi) = \Lambda$ for all $\xi \in E$ . 
Then $\Pi_\Lambda$ is simply the projection of $\Pi$ onto $\mathbb{L}$.
Note that $\Pi_\Lambda$ is independent of $\eta$ and that $\Pi_\alpha \cup \Pi_\beta \subset \Pi_\Lambda$.

We state next a monotonicity property.
\begin{lemma}[Monotonicity]
\label{l:monot}
Let $\alpha_i, \beta_i:E\to \R_+$, $i=1,2$, be measurable functions satisfying 
$\alpha_i(\xi) + \beta_i(\xi) \leq \Lambda$, $\alpha_1(\xi) \geq \alpha_2(\xi)$
and $\beta_1(\xi) \leq \beta_2(\xi)$ for all $\xi \in E$.
The following holds $\PP$-almost surely:
For any $y_i = (x_i, s) \in \LL$, $i=1,2$, such that $x_1 \geq x_2$, 
denote by $X^{i}$ the analogue of $X^{y_i}$ constructed from $\alpha_i, \beta_i$
and the same realization of $\eta, \Pi$.
Then $X^{1}_t \geq X^{2}_t$ $\forall$ $t\geq 0$.
\end{lemma}
\begin{proof}
Note that $\left(\Pi_\Lambda \cap \{x\} \times \R_+\right)_{x \in \Z}$ are i.i.d.\ Poisson point processes on $\R_+$ with intensity $\Lambda \dd t$,
and thus almost surely any two of them have empty intersection.
On this event, because $X^i$ only performs nearest-neighbor jumps, 
if $X^1_{u} > X^2_u$ and $X^1_{v} < X^{2}_v$ with $u < v$ there must be a $w \in (u,v)$ such that $X^1_w=X^2_w$.
Thus, since $X^1_0 \geq X^2_0$, we only need to prove that, if $X_t^1 = X_t^2 = x$ for some $t \geq 0$, 
then $X_\tau^1 \geq X_\tau^2$ where $\tau = \inf\{s>t \colon X^1_s \neq x \text{ or } X^2_s \neq x \}$.
Since this is automatic if $X^2_\tau = x-1$, we may assume that $X^2_\tau \geq x$. The latter is equivalent to $(x,\tau) \notin \Pi_{\beta_2}$, 
and since $\Pi_{\beta_1} \subset \Pi_{\beta_2}$, $X^1_\tau \geq x$ as well. 
Finally, $X^{2}_\tau = x+1$ is equivalent to $(x,\tau) \in \Pi_{\alpha_2}$ and $\Pi_{\alpha_1} \supset \Pi_{\alpha_2}$, so in this case $X^1_\tau=x+1$ as well.
\end{proof}

In particular, Lemma~\ref{l:monot} implies:
\begin{enumerate}
\item[] (Monotonicity in $y$). $\;$ $\PP$-a.s., if $x_1 \geq x_2$ then $X^{(x_1, s)}_t \geq X^{(x_2, s)}_t$ for all $s,t\geq 0$.
\end{enumerate}

Another simple but useful consequence of Lemma~\ref{l:monot} is the following.
Define
\begin{equation}\label{d:lambda}
\lambda = 2\Lambda.
\end{equation}
\begin{lemma}[Bound on the increments]
\label{l:bdincr}
There exists on $(\Omega, \cF, \PP)$ a collection $(N^y)_{y \in \LL}$ of Poisson processes with rate $\lambda=2\Lambda$ such that
\begin{enumerate}
\item[a)] $\PP$-a.s., $\sup_{s \in [0,t]} |X^y_s| \leq N^y_t$ for all $t \geq 0$ and $y \in \LL$;
\item[b)] For each $s \geq 0$, the collection $(N^{(x,s)})_{x \in \Z}$ is independent of $\sigma(\eta, \Pi \cap \Z \times [0, s] \times [0,\Lambda])$.
\end{enumerate}
\end{lemma}
\begin{proof}
Denote by $Z^{y, +}$, $Z^{y, -}$ the analogues of $X^y$ for $(\alpha, \beta) \equiv (\Lambda, 0)$ and $(\alpha, \beta) \equiv (0, \Lambda)$, respectively.
Note that $\pm Z^{y,\pm}$ are independent Poisson processes with rate $\Lambda$ and, by Lemma~\ref{l:monot}, $\PP$-a.s.\ $Z^{y,-}_t \leq Z^{y,-}_s \leq X^y_s \leq Z^{y,+}_s \leq Z^{y,+}_t$ for any $0 \leq s \leq t$.
To conclude, note that $N^y_t = Z^{y,+}_t - Z^{y,-}_t$ satisfies the requirements of the statement since
$(N^{(x,s)})_{x \in \Z}$ is measurable in $\sigma(\Pi \cap \Z \times (s,\infty) \times [0,\Lambda])$.
\end{proof}

It will be convenient to extend the decoupling inequality \eqref{e:DEC} to random variables in the enlarged probability space.
To this end, we redefine next the notion of support.
\begin{definition}
\label{d:support}
Let $R \subset \R^2$.
We say that a random variable on $(\Omega, \cF, \PP)$ is \emph{supported on $R$} if 
it is measurable in $\sigma\big( (\eta_t(x))_{(x,t) \in R \cap \LL}, \Pi \cap (R \times [0,\Lambda])\big)$.
\end{definition}

For future reference, we state here the following.
\begin{remark}
\label{r:ext_decouple}
Since $\Pi$ is independent of $\eta$ and also independent in disjoint space-time regions,
if the Decoupling Assumption \eqref{e:DEC} holds then it remains valid after enlarging the probability space and redefining the notion of support as in Definition~\ref{d:support}.
\end{remark}

\section{Proof overview}
\label{s:proofoverview}

\par In this section we define key objects that will help us implement our renormalization scheme and state as intermediate results the main ingredients in the proof of Theorem~\ref{t:maingeneral}. Henceforth we assume that the hypotheses of the latter are in force.

Recall $\lambda = 2 \Lambda$ where $\Lambda$ is as in \eqref{e:BR}.
By Lemma~\ref{l:bdincr}, this will help us bound the number jumps of the random walk in a given time interval.

We define next a key event for our analysis. 
For $H \geq 1$, $v \in \R$ and $w \in \R \times \R_+$, let
\begin{equation}
\label{e:defA_H}
  A_{H,w}(v)  := \Big[ \exists \, y \in I_H(w) \cap \LL : X^y_{H} - \pi_1(y) \geq v H \Big]
\end{equation}
where
\begin{equation}\label{e:defI_H}
I_H(w) = w+[0,\lambda H)\times\{0\}.
\end{equation}

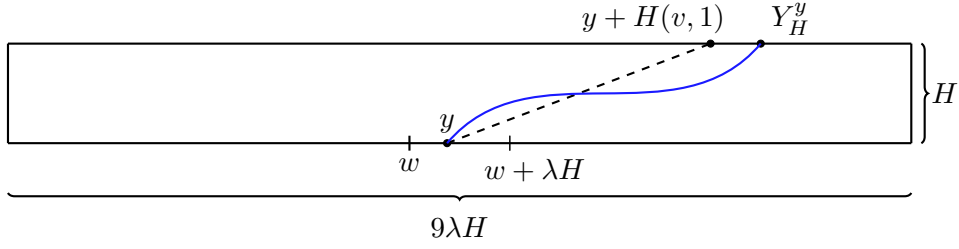
\begin{figure}
\begin{center}
\begin{tikzpicture}[scale=0.66]


\node[below] at (8,-0.1) {$w$};
\node[below] at (10.5,-0.1) {$w+\lambda H$};
\node[above] at (8.75,0) {$y$};
\node[left] at (14.5,2.5) {$y+H(v,1)$};
\node[right] at (15,2.5) {$Y^y_H$};

\draw[-,thick] (8,-0.15)--(8,0.15);
\filldraw (10,-0.15)--(10,0.15);
\filldraw (8.75,0) circle (2pt);
\filldraw (14,2) circle (2pt);
\filldraw (15,2) circle (2pt);

\draw[-, thick]  (0,0) -- (0,2);
\draw[-, thick]  (0,0) -- (18,0);
\draw[-, thick]  (18,0) -- (18,2);
\draw[-, thick]  (0,2) -- (18,2);
\draw [-, thick, dashed]  (8.75,0) -- (14,2);
\draw [blue!90!] [-,thick] (8.75,0) to[out=50, in=230] (15,2);
\draw [thick,decoration={brace,mirror},decorate] (0,-1) -- (18,-1) node[midway,below,yshift=-.17cm] {$9\lambda H$};
\draw [thick,decoration={brace,mirror},decorate] (18.2,0) -- (18.2,2);
\node[right] at (18.2,1) {$H$};
\end{tikzpicture}
\caption{\small{Starting from the point $y\in I_H(w) \cap \LL$, the random walk (blue) attains an average speed larger than $v$ during the time interval $[0,H]$.}}
\label{fig:key_event}
\end{center}
\end{figure}

In words, $A_{H,w}(v)$ is the event where it is possible to find at least one starting position $y$ inside a given space interval of length $\lambda H$ such that the random walk $X^{y}$ attains average speed at least $v$ over a time interval of length $H$. See Figure~\ref{fig:key_event} for an illustration.

The probability of $A_{H,w}(v)$ may depend on the reference point $w$. 
In fact, if $w$ does not belong to $\LL$, the cardinalities of $I_H(w) \cap\mathbb{L}$ and $I_H(0) \cap\mathbb{L}$ may differ by at most two.
In order to get a quantity that does not depend on $w$, we define
\begin{equation} 
\label{p_H}
  p_H(v) :=\sup_{w\in \R \times \R_+}  \PP \big[A_{H,w}(v) \big]= \sup_{w \in [0,1) \times \{0\}}  \PP \big[A_{H,w}(v) \big],
\end{equation}
where the second equality follows by translation invariance.  

These definitions allow us to set
\begin{equation}
\label{def_v_+}
v_+ := \inf \big\{ v \in \mathbb{R} \colon \liminf_{H \to \infty} p_H (v) = 0 \big\},
\end{equation}
which we call the {\it upper speed} of $X$.

Similarly, we define, for $w \in \R \times \R_+$, $v \in \R$ and $H \geq 1$, the event
\begin{equation*}
  \tilde{A}_{H,w}(v):=\Big[ \exists \, y \in I_H(w) \cap \mathbb{L} : X^y_{H} - \pi_1(y) \leq v H \Big],
\end{equation*}
as well as
\begin{equation*}
\tilde{p}_H(v) :=  \sup_{w\in\R \times \R_+ }\PP \big[ \tilde{A}_{H,w}(v) \big]=\sup_{w \in [0,1) \times \{0\} }\PP \big[\tilde{A}_{H,w}(v) \big]
\end{equation*}
and the {\it lower speed} of $X$ 
\begin{equation}\label{def_v_-}
v_- := \sup \big\{ v \in \mathbb{R} \colon \liminf_{H \to \infty} \tilde{p}_H(v) = 0 \big\}.
\end{equation}

The next result provides crude bounds on $v_{+}$ and $v_{-}$,
which are intuitive by Assumption~\eqref{e:BAL} and Lemma~\ref{l:bdincr}.
Its proof will be given at the end of this section.
 
\begin{proposition}\label{p:bounds_to_vs}
$v_{-}, v_{+} \in [v_{\star}, \lambda]$.
\end{proposition}

The proof of Theorem~\ref{t:maingeneral} will be done in two steps. 
The first is to obtain large deviation bounds involving speeds above $v_+$ or below $v_-$, as stated next.
Recall the definition of $\kappa, \gamma$ in \eqref{e:defgammakappa}.

\newconstant{c:c_accommodate}

\begin{proposition}
\label{prop:decay_v+-}
For every $\varepsilon > 0$, there exists a positive constant $\useconstant{c:c_accommodate}= \useconstant{c:c_accommodate}(\varepsilon)$ 
such that
  \begin{equation}
  \label{e:decay_v+-}
    \begin{split}
      p_H(v_+ + \varepsilon) \leq \useconstant{c:c_accommodate}e^{-3\kappa \log^{\gamma}H} \quad \text{ and } \quad \tilde{p}_H (v_- - \varepsilon) \leq \useconstant{c:c_accommodate}e^{-3\kappa \log^{\gamma}H}
    \end{split}
  \end{equation}
 for all  $H \geq 1$.
  \end{proposition}

The main tool for the proof of Proposition~\ref{prop:decay_v+-} is a multiscale renormalization scheme 
that allows us to first conclude \eqref{e:decay_v+-} along a specific subsequence of $H$.
An interpolation argument then yields the result for $H\geq 1$. 

Although it is intuitive to expect that $v_- \leq v_+$, we are unable to prove this directly from the definitions.
It is however a consequence of Proposition~\ref{prop:decay_v+-}:
\begin{corollary}\label{cv+mv-}
$v_- \leq v_+$.
\end{corollary}
The proof of Corollary~\ref{cv+mv-} can be found in \cite[Corollary~3.3]{HKA19}.
For completeness, we include it after the proof of 
Proposition~\ref{prop:decay_v+-} in Section~\ref{ss:interp}.

The second step in the proof of Theorem~\ref{t:maingeneral} is to show that $v_-$ and $v_+$ coincide. 
Their common value $v$ is the candidate for the speed in~\eqref{e:genLLN}.

\begin{proposition}\label{v+=v-}
$v_{+}=v_{-}$.
\end{proposition}

Proposition \ref{v+=v-} will be proved in Section~\ref{s:v-=v+}.
The idea is roughly as follows. 
Suppose $v_{-} \neq v_{+}$, which by Corollary \ref{cv+mv-} amounts to assume $v_{-}<v_{+}$. 
Definitions \eqref{def_v_+}--\eqref{def_v_-} imply that $X$ must reach average speeds between $v_{-}$ and $v_{+}$ and very close to each of them with positive probability.
Suppose $X$ spends a non-trivial part of its time running with average speed just slightly above $v_{-}$.
This will introduce delays in comparison to $v_+$ which,
since $X$ must also reach average speeds very close to $v_+$, must be compensated by moving faster than $v_+$ for a good amount of time.
However, Proposition \ref{prop:decay_v+-} tells us that this is very unlikely.
Using renormalization arguments, this intuition can be made rigorous, 
leading to a contradiction that implies $v_{-}= v_{+}$.

With Propositions \ref{prop:decay_v+-} and \ref{v+=v-} in hand, the proof of Theorem~\ref{t:maingeneral} can be concluded using a Borel-Cantelli argument.
This will be done in Section~\ref{s:proof_main}.

We end this section with the proof of Proposition~\ref{p:bounds_to_vs}.
\begin{proof}[Proof of Proposition \ref{p:bounds_to_vs}]
We first show that $v_{\star} \leq v_{-} \leq \lambda$. 
Observe that a union bound together with \eqref{e:BAL} and translation invariance imply
\begin{equation*}
\begin{split}
\PP \left[ \tilde{A}_{H,w} (v_{\star}) \right] 
= \PP \left[ \exists \, y \in I_H(w)\cap \mathbb{L} : X^y_{H} - \pi_1(y) \leq v_{\star} H \right]
\leq C_\star\lambda H e^{-\kappa_\star(\log H)^{\gamma_\star}}.
\end{split}
\end{equation*}
Taking the supremum over $w$ we obtain $\liminf_{H \to \infty} \tilde{p}_H(v_{\star}) = 0$, implying by the definition that $v_{-} \geq v_{\star}$.
For the upper bound, fix $\varepsilon>0$ and notice that
\begin{equation*}
\begin{split}
\PP \left[ \tilde{A}_{H,w}(\lambda+\varepsilon) \right] & = \PP \Big[ \exists \, y \in I_H(w)\cap \mathbb{L} : X^y_{H} - \pi_1(y) \leq (\lambda+\varepsilon) H \Big]\\
& \geq \PP \left[X_{H} \leq (\lambda+\varepsilon) H\right]
= 1 - \PP \left[X_{H} > (\lambda+\varepsilon) H\right].
\end{split}
\end{equation*}
By Lemma~\ref{l:bdincr}, $|X_H|$ is stochastically dominated by a Poisson random variable $N_H$ with parameter $\lambda H$. Choosing $\mu>0$ small enough we obtain
\begin{equation}\label{eq:poisson_domination}
\PP \left[X_{H} > (\lambda+\varepsilon) H\right] \leq \mathbb{E}\big[ e^{\mu N_{H}}\big] e^{-(\lambda+\varepsilon)\mu H} = e^{\lambda H (e^{\mu}-1)}e^{-(\lambda+\varepsilon)\mu H} \leq e^{-c H}
\end{equation}
for some $c >0$. Taking the supremum in $w$, we obtain $\lim_{H \to \infty}\tilde{p}_{H}(\lambda+\varepsilon) = 1$ and thus $v_{-} \leq \lambda+\varepsilon$. Finally, since $\varepsilon > 0$ is arbitrary, $v_{-} \leq \lambda$.

Let us now focus on $v_{+}$. For the lower bound, observe that
\begin{equation*}
\begin{split}
\PP \left[A_{H,w} (v_{\star}) \right] & = \PP \left[ \exists \, y \in I_H(w)\cap \mathbb{L} : X^y_{H} - \pi_1(y) \geq v_{\star} H \right]\\
& \geq \PP \left[ X_{H} > v_{\star} H\right] = 1- \PP \left[X_{H} \leq v_{\star} H \right]\\
&\geq 1 - C_\star e^{-\kappa_\star\log^{\gamma_\star} H}.
\end{split}
\end{equation*}
This implies $\liminf_{H \to \infty} p_{H}(v_{\star})=1$, and thus $v_{+} \geq v_{\star}$.
For the upper bound, we proceed as in~\eqref{eq:poisson_domination}. For any $\varepsilon>0$, there exists $c>0$ such that
\begin{equation*}
\PP\left[A_{H,w}(\lambda+\varepsilon)\right] \leq \lambda H e^{- c H},
\end{equation*}
implying $\liminf_{H \to \infty}p_{H}(\lambda+\varepsilon)=0$ for every $\varepsilon>0$ and thus $v_{+} \leq \lambda$.
\end{proof}

\section{Proof of Proposition \ref{prop:decay_v+-}}
\label{s:Proof_decaiv}

In this section we prove Proposition \ref{prop:decay_v+-} using a renormalization approach in which we adapt and combine ideas from \cite{blondel2020random,hilario2015random,HKA19}.
We divide this section into three parts.
Section~\ref{sub_renor_1} contains the definitions of the scales and boxes that support the main events that we will consider. 
In Section~\ref{ss:decay_particular_sequence} we prove Proposition~\ref{prop:decay_v+-} for values of $H$ that are multiples of our scales. 
In Section~\ref{ss:interp} an interpolation argument is used to extend the result to $H \geq 1$. 
We conclude with the proof of Corollary~\ref{cv+mv-}.

\subsection{Scales and boxes}\label{sub_renor_1}

\newconstant{c:increasing_L_k}

\par We begin by introducing a sequence of scales $(L_k)_{k \geq 0}$ inductively as
\begin{equation}\label{def_scales_rec}
  L_0 := 10^{10} \quad \text{and} \quad L_{k + 1} := \ell_k L_k, \quad \text{for } k \geq 0,
\end{equation}
where $\ell_k := \lfloor L_k^{\nu} \rfloor$ and $\nu \in (0, 1)$ is chosen such that, with $\gamma$ as in \eqref{e:defgammakappa},
\begin{equation}\label{e:choice_nu}
6(1+\nu)^{3\gamma} \leq 7.
\end{equation}
Note that $L_k$ grows super-exponentially fast and, for some $\useconstant{c:increasing_L_k}>0$,
\begin{equation}\label{increasing_L_k}
\useconstant{c:increasing_L_k} L_{k}^{1+\nu} \leq L_{k+1} \leq L_{k}^{1+\nu}, \quad \text{for every } k \geq 0.
\end{equation}

Given an integer $L \in \N$ and a real number $h\geq 1$, we define the box 
\begin{equation*}
B_{hL}:=[-4\lambda h L, 5\lambda h L)\times [0, h L)\subset \mathbb{R}^2,
\end{equation*}
(recall $\lambda=2\Lambda$ with $\Lambda$ as in \eqref{e:BR}). Similarly to \eqref{e:defI_H}, we set
\begin{equation*}
I_{hL}=[0, \lambda h L) \times \{0\}\subset \mathbb{R}^2
\end{equation*}
to be the middle interval of length $\lambda h L$ contained in bottom face of $B_L^h$.

The value $L$ in the definitions above will frequently be replaced by $L_k$ to apply renormalization. The factor $\lambda$ ensures that any trajectory starting at $I_{hL}$ remains inside $B_{hL}$ up to the time $hL$ with high probability. The parameter $h$ is a scaling factor that will be chosen later; its role will become clearer in the proofs (see in particular the proof of Lemma~\ref{trigger} and preceding comments).

For $w \in \mathbb{R}^2$, we write the translations of $B_{hL}$ and $I_{hL}$ by $w$ as
\begin{equation*}
B_{hL}(w):=w + B_{hL} \quad \text{and}\quad I_{hL}(w)= w + I_{hL}.
\end{equation*} 
We denote 
\begin{equation*}
  M^h_k := \{h\} \times \{k\} \times \mathbb{R}^2
\end{equation*}
and for $m = (h,k,w) \in M^h_k$ and  $v\in\mathbb{R}$ we write
\begin{equation*}
B_m := B_{h L_k}(w),\text{  } I_m := I_{h L_k}(w), \text{  } A_m(v):=A_{hL_k,w}(v) \text{ and } \tilde{A}_m(v):=\tilde{A}_{hL_k,w}(v),
\end{equation*}
with $A_{H,w}(v)$, $\tilde{A}_{H,w}(v)$ as defined in Section~\ref{s:proofoverview}.

To relate events in consecutive scales, it will be useful to define index sets as follows.
For $m = (h,k+1,(z,s)) \in M^{h}_{k+1}$, define
\begin{equation}\label{e:Im}
\begin{split}
\cI_{m}= 
\bigg\{ & (h,k, (z+ i \lambda h L_k, s+ j \lambda hL_k) ) \in M^h_k \colon \\
& (i,j) \in [-4 \ell_k, 5 \ell_k -1] \times [0, \ell_k-1] \cap \Z^2 \bigg\}.
\end{split}
\end{equation}
Note that $|\cI_m|\leq9\ell_k^2$.
A simple but important observation is that
\begin{equation}\label{e:conseqIm}
\text{ if } y = (i, s+ j hL_k) \in B_m \text{ with } i \in \Z, j \in \N_0 \text{ then } \exists\, m' \in \cI_m \text{ such that }y \in I_{m'}.
\end{equation}

\subsection{Decay of $p_H(v)$, $\tilde{p}_H(v)$ along a particular sequence}
\label{ss:decay_particular_sequence}

In this section, we define a multiscale renormalization scheme to prove \eqref{e:decay_v+-} for a sequence of the form $H=h L_k$. 
A key ingredient in the argument is the decoupling inequality \eqref{e:DEC}.
In order to apply it, we need to ensure that certain relevant events are supported in well-separated boxes. 
To this end, we will use Lemma~\ref{l:bdincr} and Assumption~\eqref{e:BAL} with $v_\star>v_\circ$ to restrict to an event of large probability where the random walk is well-behaved. To describe this event, define for $m = (h,k+1, (z,t))$ the set
\begin{equation}\label{Cm}
\cC_{m}= \left\{ y = (x,s) \in B_m \colon\, x \in \Z, s-t \in h L_k \N_0 \right\}.
\end{equation}
whose cardinality is bounded by $9 (h L_k)^3$. Then put
\begin{equation}\label{control_random_walk}
D_{m}= \hat{D}_{m} \cap \bar{D}_{m}
\end{equation}
where
\begin{equation}\label{e:rw_inside_Bm}
\hat{D}_{m}=\bigcap_{y \in \cC_m} \Big\{\sup_{s \in [0,hL_k]}|X^y_{s}-\pi_{1}(y)|\leq 4 \lambda hL_{k} \Big\}
\end{equation}
and 
\begin{equation}\label{e:min_displ_Bm}
\bar{D}_{m}= \bigcap_{y \in \cC_m}\Big\{X^{y}_{hL_{k}}-\pi_{1}(y)> v_{\star} hL_{k}\Big\}.
\end{equation} 
Note that, on $\hat{D}_{m}$, if $y \in I_{m^\prime} \cap \cC_m$ for some $m^{\prime} \in M^{h}_{k}$ then $X^y$ remains inside $B_{m^\prime}$ for a time interval of length $hL_{k}$. Consequently, if $y \in I_m \cap \LL$ then $X^y$ does not leave $B_m$ before time $hL_{k+1}$.
To bound $\PP[\hat{D}_m^c]$, we use Lemmas~\ref{l:bdincr} and \ref{l:PoiConc} together with a union bound and translation invariance to obtain
\begin{equation}\label{event_one}
\sup_{m \in M^h_{k+1}} \PP [\hat{D}_{m}^{c}]\leq 9 \lambda (h L_{k})^3e^{-2\lambda hL_{k}}.
\end{equation}

On $\bar{D}_{m}$, if $X$ visits two points $y_{1}, y_{2} \in C_m$ then its average speed between these visits is larger than $v_{\star}$.
The ballisticity assumption \eqref{e:BAL} implies
\begin{equation}\label{event_two}
\PP [\bar{D}_m^{c}] \leq 9 C_\star \lambda (h L_{k})^3 e^{-\kappa_\star (\log h L_{k})^{\gamma_\star}}. 
\end{equation}

\newconstant{c:deviation_1}

Putting \eqref{event_one}--\eqref{event_two} together and using \eqref{e:defgammakappa}, \eqref{control_random_walk}, we obtain the following.
\begin{lemma}\label{drifandrwbox}
There exists a constant $\useconstant{c:deviation_1}>0$ such that
\begin{equation*}
\sup_{m \in M^h_{k+1}} \PP \left[D_m^c\right] \leq \useconstant{c:deviation_1} e^{-8 \kappa (\log h L_k)^{\gamma}} \qquad \text{ for every } k \geq 0 \text{ and } h \geq 1.
\end{equation*}
\end{lemma}

\newconstant{c:h_trigger}
\newconstantk{k:after_turbo_charger}

The next result is the main goal of this section.
\begin{lemma}\label{trigger}
For every $v>v_{+}$, there exist constants $\useconstant{c:h_trigger}>0$, $\useconstantk{k:after_turbo_charger} \in \N$ such that
\begin{equation}\label{eq_trigger}
p_{\useconstant{c:h_trigger}L_k}(v)\leq e^{-4 \kappa \log^{\gamma} L_{k}} \qquad \text{ for all } k \geq \useconstantk{k:after_turbo_charger}.
\end{equation}
Analogously, for every $\tilde{v}<v_-$, there exist $\tilde{\useconstant{c:h_trigger}}>0$, $\tilde{\useconstantk{k:after_turbo_charger}} \in \N$ such that
\begin{equation}\label{eq_trigger_tilde}
\tilde{p}_{\tilde{\useconstant{c:h_trigger}}L_k} (\tilde{v})\leq e^{-4 \kappa \log^{\gamma} L_{k}} \qquad \text{ for all } k \geq \tilde{\useconstantk{k:after_turbo_charger}}.
\end{equation}
\end{lemma}

Lemma~\ref{trigger} will be proved via renormalization. 
The idea consists in recursively transferring estimates for the probability of $A_m$, $\tilde{A}_m$ from one scale to the next.
For this, the speeds considered in each scale must change slightly in order to accommodate small errors.
We now present the deterministic part of this argument for the case of $A_m$, which is inspired by \cite[Lemma 5.2]{HKA19}.

\newconstantk{k:scale_to_boxes}

\begin{lemma}\label{l:cascading_boxes}
There exists an integer $\useconstantk{k:scale_to_boxes}\in \N$ such that the following holds for any $k \geq \useconstantk{k:scale_to_boxes}$.
Fix two speeds $0<v_{\min}<v_{\max}$ and set $\bar{v} = v_{\min}+(v_{\max} - v_{\min})/\sqrt{\ell_k}$.
Then, for any $h\geq1$ and any $m \in M^h_{k+1}$, one of the following is true:
\begin{enumerate}
\item[a)] $A_m(\bar{v}) \cap D_m$ does not occur;
\item[b)] $A_{m'}(v_{\max})$ occurs for some $m' \in \cI_m$;
\item[c)] there exist two indices $m_1$, $m_2 \in \cI_m$ such that $A_{m_{1}}(v_{\min})\cap A_{m_{2}}(v_{\min})$ occurs and $\dist_{H}\geq (v_\circ \dist_{V}+\usebigconstant{c:decoupling_2} hL_k + \usebigconstant{c:decoupling_3}) \vee \lambda hL_k$, where $\dist_{H}$ and $\dist_{V}$ are the horizontal and vertical distances of the boxes $B_{m_{1}}$ and $B_{m_{2}}$.
\end{enumerate}
\end{lemma}
\begin{proof} 
Write $\delta_\star := v_{\star}-v_\circ>0$ and fix $N, \useconstantk{k:scale_to_boxes} \in \N$ such that
\begin{equation}
\label{e:assumptionN}
\delta_\star N \geq 11\lambda + \usebigconstant{c:decoupling_2} + \usebigconstant{c:decoupling_3}
\quad \text{ and } \quad \frac{N}{\ell_k}<\frac{1}{\sqrt{\ell_k}} \; \text{ for all } k \geq \useconstantk{k:scale_to_boxes}.
\end{equation}
Fix $k \geq \useconstantk{k:scale_to_boxes}$, $h \geq 1$, $m \in M^h_{k+1}$ and assume that both $a)$ and $b)$ fail,
i.e., $A_{m}(\bar{v})\cap D_{m}$ occurs and $\cup_{m' \in \cI_m}A_{m'}(v_{\max})$ does not (recall \eqref{control_random_walk}).
We claim that
\begin{equation}\label{claim_cascading}
\begin{array}{cc}
&\text{there exist $y \in I_m \cap \LL$ and at least $N+1$ elements}\\
&\text{$m_{i} = (h,k,(x_{i}, s_{i})) \in \cI_m$, $1 \leq i\leq N+1$, with $s_i \neq s_j$ when $i \neq j$,}\\
&\text{such that $A_{m_i} (v_{\min})$ occurs and $X^y$ visits $I_{m_i}$ for all $1 \leq i \leq N+1$.}\\
\end{array}
\end{equation}
We will prove \eqref{claim_cascading} by contradiction. 
Suppose that it is false.
Then, starting from any $y \in I_m \cap \LL$,  
we split the displacement of $X^y$ during time $hL_{k+1}$ into the sum of displacements in time intervals of length $hL_k$ to obtain
\begin{equation}\label{cascading_displacement_1}
\begin{split}
X^{y}_{h L_{k+1}}-\pi_{1}(y) 
&= \sum_{j=0}^{\ell_k-1}\Big[X^{Y^{y}_{jhL_{k}}}_{h L_{k}}-X^{y}_{jh L_{k}}\Big] 
 \leq \left[\ell_k - N\right]hL_k v_{\min} + N h L_{k} v_{\max}\\
& = h L_{k+1} \Big\{ v_{\min} + (v_{\max} - v_{\min}) \frac{N}{\ell_k} \Big\}
< h L_{k+1} \bar{v},
\end{split}
\end{equation}
where for the first inequality we used that $\hat{D}_m$ occurs to conclude $Y^y_{jhL_k} \in B_m$ for every $0 \leq j \leq \ell_{k}-1$ so that we can apply \eqref{e:conseqIm} and the definition of $A_{m'}$, and for the second we used \eqref{e:assumptionN}.
This implies that $A_{m}(\bar{v})$ does not occur, proving \eqref{claim_cascading}. 

Let us assume that $m_1, \ldots, m_{N+1}$ are indexed so that the time coordinate of $I_{m_1}$ is the smallest and that of $I_{m_2}$ is the largest. 
Hence ${m_{1}}$ and ${m_2}$ are the first and the last of the $N+1$ elements in \eqref{claim_cascading}  for which $A_{m_{i}}(v_k)$ occurs, as illustrated in Figure \ref{fig:cascading_boxes}.

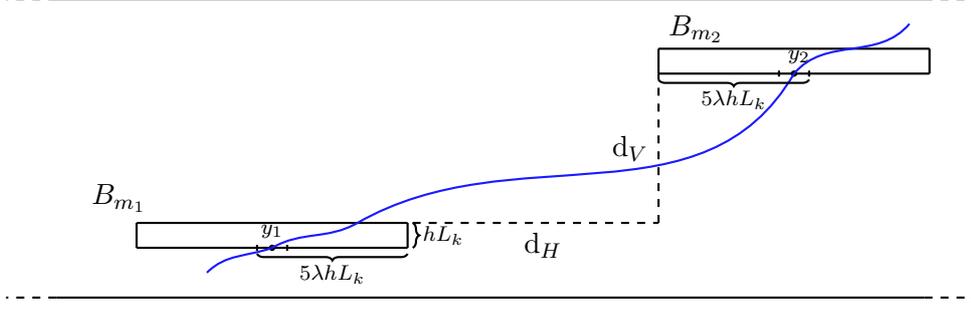
\begin{figure}
\begin{center}
\begin{tikzpicture}[scale=0.66]


\draw[-, thick]  (2.6,1) -- (2.6,1.5);
\draw[-, thick]  (2.6,1) -- (8,1);
\draw[-, thick]  (2.6,1.5) -- (8,1.5);
\draw[-, thick]  (8,1) -- (8,1.5);

\draw[-, thick]  (13,4.5) -- (13,5);
\draw[-, thick]  (13,4.5) -- (18.4,4.5);
\draw[-, thick]  (13,5) -- (18.4,5);
\draw[-, thick]  (18.4,4.5) -- (18.4,5);

\draw[-, thick]  (1,0) -- (18.3,0);
\draw[-, thick]  (1,6) -- (18.3,6);
\draw[dashed, thick]  (1,0) -- (0,0);
\draw[dashed, thick]  (1,6) -- (0,6);
\draw[dashed, thick]  (18.3,0) -- (19.3,0);
\draw[dashed, thick]  (18.3,6) -- (19.3,6);

\filldraw (5.3,1) circle (1.5pt);
\filldraw (15.7,4.5) circle (1.5pt);
\node[above] at (5.3,1) {\scriptsize{$y_{1}$}};
\node[above] at (15.8,4.5) {\scriptsize{$y_{2}$}};
\node[below] at (10.7,1.5) {$\dist_{H}$};
\node[left] at (13,3) {$\dist_{V}$};
\node[left] at (3,2) {$B_{m_{1}}$};
\node[right] at (13,5.4) {$B_{m_{2}}$};


\draw[-, thick]  (5,0.94) -- (5,1.06);
\draw[-, thick]  (5.6,0.94) -- (5.6,1.06);
\draw[-, thick]  (15.4,4.44) -- (15.4,4.56);
\draw[-, thick]  (16,4.44) -- (16,4.56);
\draw [blue!90!] [-,thick] (4,0.5) to[out=45, in=200] (5.3,1);
\draw [blue!90!] [-,thick] (5.3,1) to[out=30, in=210] (7,1.5);
\draw [blue!90!] [-,thick] (7,1.5) to[out=30, in=240] (15.7,4.5);
\draw [blue!90!] [-,thick] (15.7,4.5) to[out=50, in=230] (18,5.5);
\draw[dashed, thick]  (8.1,1.5) -- (13,1.5);
\draw[dashed, thick]  (13,1.5) -- (13,4.5);
\draw[thick,decoration={brace,mirror},decorate] (5,0.88) -- (8,0.88) node[midway,below,yshift=-.005cm] {\scriptsize{$5\lambda hL_{k}$}};
\draw[thick,decoration={brace,mirror},decorate] (13,4.38) -- (16,4.38) node[midway,below,yshift=-.005cm] {\scriptsize{$5\lambda hL_{k}$}};
\draw [thick,decoration={brace,mirror},decorate]  (8.1,1) -- (8.1,1.5); 
\node[right] at (8.1, 1.25) {\scriptsize{$ h L_k$}};
\end{tikzpicture}

\caption{\small{The boxes and $B_{m_{1}}$ and $B_{m_2}$ contained in $B_{m}$ and the points $y_{1} \in I_{m_1}$ and $y_2 \in I_{m_2}$.}}\label{fig:cascading_boxes}
\end{center}
\end{figure}

Let $\dist_V$, $\dist_H$ be the vertical and the horizontal distances of $B_{m_{1}}$, $B_{m_2}$. 
By \eqref{claim_cascading}, there exist $y_{1} \in I_{m_{1}}\cap \mathbb{Z}^2$, $y_2 \in I_{m_2}\cap \mathbb{Z}^2$ such that $Y^{y_1}_{\dist_V + hL_k}=y_2$.
Note that $d_V + hL_k = \hat{N} h L_k$ for some integer $\hat{N} \geq N$.
Since we are on $\bar{D}_m$,
\begin{equation}\label{bound_to_displacement_1}
\begin{aligned}
& \pi_{1}(y_2)-\pi_{1}(y_1)  = X_{\hat{N} h L_k}^{y_1} - \pi_1(y_1)
= \sum_{j=0}^{\hat{N}-1} X^{Y^{y_1}_{j h L_k}}_{h L_k} - X^{y_1}_{j h L_k}
\\
& \geq v_\star \hat{N} h L_k = v_{\star}(\dist_{V}+hL_{k})= (v_\circ+\delta_\star)(\dist_{V}+hL_{k}) > v_\circ \dist_V + \delta_\star N h L_k.
\end{aligned}
\end{equation}
On the other hand, by the geometry of the boxes $B_{m_i}$ and since $y_i \in I_{m_i}$, $i=1,2$, 
\begin{equation}\label{bound_to_displacement_2}
\pi_{1}(y_2)-\pi_{1}(y_1)\leq \dist_{H}+10 \lambda h L_{k}.
\end{equation}
Combining~\eqref{bound_to_displacement_1}--\eqref{bound_to_displacement_2} and using \eqref{e:assumptionN}, we obtain
\begin{equation}\label{bound_dh_1}
\begin{aligned}
\dist_{H} 
& > v_\circ \dist_V + \delta_\star N h L_k - 10\lambda h L_k 
 \geq v_\circ \dist_{V} + (\lambda+\usebigconstant{c:decoupling_2}+\usebigconstant{c:decoupling_3}) h L_{k} \\
& \geq v_\circ \dist_V + \usebigconstant{c:decoupling_2} h L_k + \usebigconstant{c:decoupling_3} + \lambda h L_k
\geq (v_\circ \dist_V + \usebigconstant{c:decoupling_2} h L_k + \usebigconstant{c:decoupling_3})\vee \lambda h L_k
\end{aligned}
\end{equation}
as claimed, where for the third inequality we used that $h L_k \geq 1$.
\end{proof}

\newconstant{c:general_recursion}

As a consequence we obtain the following result.
\begin{lemma}\label{l:recgen}
There exists a constant $\useconstant{c:general_recursion} > 0$ such that, for any $0<v_{\min} < v_{\max}$, any $h \geq 1$ and
any $k \geq \useconstantk{k:scale_to_boxes}$,
\begin{equation}\label{e:recgen1}
p_{hL_{k+1}}\Big(v_{\min} + \frac{v_{\max}-v_{\min}}{\sqrt{\ell_k}}\Big) 
\leq \useconstant{c:general_recursion} \ell_k^4 \left\{ p_{h Lk}(v_{\min})^2 + p_{h L_k}(v_{\max}) + e^{-8\kappa (\log h L_k)^\gamma}\right\}
\end{equation}
and the term $p_{hL_k}(v_{\max})$ may be omitted if $v_{\max} > 4 \lambda$. Analogously,
\begin{equation}\label{e:recgen2}
\tilde{p}_{hL_{k+1}}\Big(v_{\max} - \frac{v_{\max}-v_{\min}}{\sqrt{\ell_k}}\Big) 
\leq \useconstant{c:general_recursion} \ell_k^4 \left\{ \tilde{p}_{h Lk}(v_{\max})^2 + \tilde{p}_{h L_k}(v_{\min}) + e^{-8\kappa (\log h L_k)^\gamma}\right\}
\end{equation}
and the term $\tilde{p}_{hL_k}(v_{\min})$ may be omitted if $v_{\min} \leq v_\star$.
\end{lemma}
\begin{proof}
Here we will only prove \eqref{e:recgen1}. The proof of \eqref{e:recgen2} can be obtained similarly using the analogous of Lemma~\ref{l:cascading_boxes} for $\tilde{A}_m$. The details are left for the reader.

Fix $k\geq\useconstantk{k:scale_to_boxes}$, $h \geq 1$, $m \in M^{h}_{k+1}$, $v_{\min} < v_{\max}$ and let $\bar{v}= v_{\min}+(v_{\max} - v_{\min})/\sqrt{\ell_k}$.
Lemma~\ref{l:cascading_boxes} together with a union bound and $|\cI_m|\leq 9 \ell_k^2$ implies
\begin{equation*}\label{e:prrec1}
\PP \left[A_{m}(\bar{v})\cap D_{m}\right] 
\leq 81 \ell_k^4 \Big\{ \sup_{\left({m_1},{m_2}\right)_{m}} \PP \left[A_{m_{1}}(v_{\min})\cap A_{m_{2}}(v_{\min}) \cap D_m\right] + p_{hL_k}(v_{\max}) \Big\},
\end{equation*}
where $\left({m_1},{m_2}\right)_{m}$ stands for the set of pairs $m_{1}, m_{2} \in \cI_m$ such that the vertical and horizontal distances of $B_{m_1}, B_{m_2}$ satisfy $\dist_{H}\geq (v_\circ \dist_{V}+ \usebigconstant{c:decoupling_2} hL_{k} +  \usebigconstant{c:decoupling_3}) \vee \lambda hL_k$.

Fix $m_1, m_2 \in \left({m_1},{m_2}\right)_{m}$.  
On $D_m$, occurrence of $A_{m_{1}}(v_{\min})$, $A_{m_{2}}(v_{\min})$ is determined by $\eta, \Pi$ inside $B_{m_{1}}$, $B_{m_{2}}$. 
Assumption~\eqref{e:DEC} and Remark~\ref{r:ext_decouple} thus give
\begin{equation*}
\begin{aligned}
 \PP\left[A_{m_{1}}(v_{\min})\cap A_{m_{2}}(v_{\min}) \cap D_m\right] 
 \leq 
 & \PP\left[A_{m_{1}}(v_{\min})\right] \PP\left[A_{m_{2}}(v_{\min})\right] \\
 & + C_\circ e^{-\kappa_\circ (\log \dist_H)^{\gamma_\circ}} + 2\PP\left[D_m^c\right].
\end{aligned}
\end{equation*}
Using Lemma~\ref{drifandrwbox}, \eqref{e:defgammakappa} and $\dist_{H} \geq \lambda h L_{k}\geq h L_{k}$, we obtain
\begin{equation}\label{imp_relation_between_scales}
\begin{split}
\PP[A_{m}(\bar{v})]&\leq \PP\left[A_{m}(\bar{v})\cap D_{m}\right] + \PP[D_{m}^{c}]\\
& \leq 81 \ell_k^4 \big\{
p_{hL_{k}}(v_{\min})^{2} + p_{h L_k}(v_{\max}) + C_\circ e^{-\kappa_\circ (\log h L_{k})^{\gamma}} + 3\PP[D_{m}^{c}] \big\}\\
& \leq \useconstant{c:general_recursion}\ell_k^4 \big[ p_{hL_{k}}^{2}(v_k) + p_{h L_k}(v_{\max}) +e^{-8 \kappa (\log h L_{k})^{\gamma}} \big]
\end{split}
\end{equation}
for an appropriate positive constant $\useconstant{c:general_recursion}$,
finishing the proof of \eqref{e:recgen1}.
To conclude, note that, increasing $\useconstant{c:general_recursion}$ if necessary, we can omit 
$p_{h L_k}(v_{\max})$ in \eqref{imp_relation_between_scales} when $v_{\max} > 4 \lambda$
since in this case $\cup_{m' \in \cI_m} A_{m'}(v_{\max}) \subset D_m^c$.
\end{proof}

Lemma~\ref{l:recgen} suggests that, in order to obtain recursive estimates for $p_{h L_k}$, $\tilde{p}_{h L_k}$,
we should work with sequences of speeds. This motivates our following definitions.

\newconstantk{k:k_to_speeds}

Consider first deviations above $v_+$.
Given $v>v_{+}$, take $\useconstantk{k:k_to_speeds}= \useconstantk{k:k_to_speeds}(v)\geq 1$ so that
\begin{equation*}
\sum_{k\geq \useconstantk{k:k_to_speeds}}\frac{1}{k^2}\leq \frac{v-v_{+}}{2}. 
\end{equation*}
Then, define recursively
\begin{equation}
\label{e:v_k}
v_{\useconstantk{k:k_to_speeds}}=\frac{v+v_{+}}{2} \text{ }\text{ and } \text{ } v_{k+1}=  v_{k} +\frac{1}{k^2}, \text{ for every $k\geq \useconstantk{k:k_to_speeds}$}.
\end{equation}
It follows that $(v_k)_{k\geq \useconstantk{k:k_to_speeds}}$ is an increasing sequence with limit $v_{\infty}\leq v$. In particular, $v_{k} \in (v_{+}, v]$ for every $k\geq \useconstantk{k:k_to_speeds}$.

Let us now consider deviations below $v_-$.
Given $\tilde{v} <v_{-}$ fix $\tilde{\useconstantk{k:k_to_speeds}} \in \N$ such that
\begin{equation*}
\sum_{k\geq \tilde{\useconstantk{k:k_to_speeds}}}\frac{1}{k^2}\leq \frac{v_{-}-\tilde{v}}{2} 
\end{equation*}
and recursively define
\begin{equation}
\label{e:v_k_tilde}
\tilde{v}_{\useconstantk{k:k_to_speeds}}=\frac{\tilde{v}+v_{-}}{2}, \quad \tilde{v}_{k+1}=  \tilde{v}_{k} -\frac{1}{k^2}, \text{ for every $k\geq \tilde{\useconstantk{k:k_to_speeds}}$}.
\end{equation}
The sequence $(\tilde{v}_k)_{k\geq \tilde{\useconstantk{k:k_to_speeds}}}$ decreases towards $\tilde{v}_{\infty}\geq \tilde{v}$. 
so that $\tilde{v}_{k} \in [\tilde{v}, v_{-})$, $k\geq \tilde{\useconstantk{k:k_to_speeds}}$.

Next we apply Lemma~\ref{l:recgen} to derive recursive bounds on $p_{hL_k}(v_k)$ and $\tilde{p}_{hL_k}(\tilde{v}_k)$.

\newconstantk{k:turbo_charger}

\begin{lemma}\label{recursive_estimate}
Fix $v>v_{+}$ and let $(v_k)$ be the sequence defined in \eqref{e:v_k}. 
There exists a positive integer $\useconstantk{k:turbo_charger}= \useconstantk{k:turbo_charger}(v)$ such that, for all $k\geq \useconstantk{k:turbo_charger}$ and all $h\geq1$,
\begin{equation}\label{Eqt:recursive_estimate}
\text{if } \, p_{hL_{k}}(v_k)\leq e^{-4\kappa \log^{\gamma} L_{k}}\,\, \text{ then } \,\, p_{hL_{k+1}}(v_{k+1})\leq e^{-4\kappa \log^{\gamma} L_{k+1}}.
\end{equation}
Analogously, if $\tilde{v}<v_{-}$ and $(\tilde{v}_k)$ is the sequence defined in \eqref{e:v_k_tilde}, there exists a positive integer 
$\tilde{\useconstantk{k:turbo_charger}}= \tilde{\useconstantk{k:turbo_charger}}(v)$ such that, for all $k\geq \tilde{\useconstantk{k:turbo_charger}}$ and all $h\geq1$,
\begin{equation}\label{Eqt:recursive_estimate_tilde}
\text{ if }\, \tilde{p}_{hL_{k}}(\tilde{v}_k)\leq e^{-4\kappa \log^{\gamma} L_{k}}\,\, \text{ then } \,\,\tilde{p}_{hL_{k+1}}(\tilde{v}_{k+1})\leq e^{-4\kappa\log^{\gamma} L_{k+1}}.
\end{equation}
\end{lemma}
\begin{proof}

We will only prove \eqref{Eqt:recursive_estimate} since the proof of \eqref{Eqt:recursive_estimate_tilde} is similar.
Let $v > v_{+}$ and consider the sequence defined in \eqref{e:v_k}. 

Note that \eqref{e:choice_nu} implies $2>(1+\nu)^{\gamma}$ and fix an integer $\useconstantk{k:turbo_charger} \geq \useconstantk{k:scale_to_boxes} \vee \useconstantk{k:k_to_speeds}$ such that
\begin{equation}\label{e:cond_turbocharger}
5\lambda k^2 < \sqrt{\ell_k} \quad \text{ and } \quad 2\useconstant{c:general_recursion}\ell_k^4e^{- [2-(1+\nu)^{\gamma}] 4 \kappa \log^{\gamma} L_{k}} \leq 1 \quad \text{for every } k \geq \useconstantk{k:turbo_charger},
\end{equation}
Fix $h \geq 1$, $k \geq \useconstantk{k:turbo_charger}$ and
assume that $p_{hL_{k}}(v_k) \leq e^{-4\kappa \log^{\gamma} L_{k}}$. 
Let $v_{\min} = v_k$, $v_{\max} = 5 \lambda$ and note that $\bar{v} = v_{\min} +(v_{\max} - v_{\min})/\sqrt{\ell_k} < v_{k+1}$ by \eqref{e:cond_turbocharger}.
Thus \eqref{e:recgen1} yields
\begin{equation*}
p_{hL_{k+1}}(v_{k + 1}) \leq p_{hL_{k+1}}(\bar{v}) \leq 2 \useconstant{c:general_recursion}\ell_k^4 e^{-8 \kappa\log^{\gamma} L_{k}}
\end{equation*}
where the first inequality holds by monotonicity.
Now note that, since $L_{k+1} \leq L_k^{1+\nu}$,
\begin{equation*}
\frac{p_{hL_{k+1}}(v_{k + 1})}{e^{-4\kappa\log^{\gamma} L_{k+1}}} \leq 2\useconstant{c:general_recursion}\ell_k^4e^{- [2-(1+\nu)^{\gamma}] 4\kappa \log^{\gamma} L_{k}} \leq 1.
\qedhere
\end{equation*}
\end{proof}

We are now ready to prove Lemma \ref{trigger}. 
The proof given next shows that the scaling parameter $h$ serves as a means to \emph{trigger} the use of Lemma~\ref{recursive_estimate}, 
i.e., the choice $h= {\useconstant{c:h_trigger}}$ with suitable ${\useconstant{c:h_trigger}}$ guarantees a desired bound on an initial scale which is then recursively transported to all higher scales by the lemma.
\begin{proof}[Proof of Lemma \ref{trigger}]
Fix $v >v_{+}$ and let $\useconstantk{k:after_turbo_charger}=\useconstantk{k:turbo_charger}(v)$. 
Recall \eqref{e:v_k}.
Since $v_{\useconstantk{k:after_turbo_charger}} > v_{+}$, we have
$\liminf_{h \to \infty} p_{hL_{\useconstantk{k:after_turbo_charger}}} (v_{\useconstantk{k:after_turbo_charger}}) = 0$.
Therefore, we can fix $\useconstant{c:h_trigger}= \useconstant{c:h_trigger}(v)\geq 1$ such that 
\begin{equation}\label{initial_eq_to trigger}
 p_{\useconstant{c:h_trigger}L_{\useconstantk{k:after_turbo_charger}}} (v_{\useconstantk{k:after_turbo_charger}}) \leq
 e^{-4\kappa\log^{\gamma} L_{\useconstantk{k:after_turbo_charger}}}.
\end{equation}
Using that $v\geq v_{k}$ and iterating \eqref{initial_eq_to trigger} through Lemma \ref{recursive_estimate}, we obtain
\begin{equation*}
p_{\useconstant{c:h_trigger}L_{k}} (v)\leq p_{\useconstant{c:h_trigger}L_{k}} (v_{k})\leq
 e^{-4\kappa\log^{\gamma} L_{k}} \quad \text{for all } k\geq \useconstantk{k:after_turbo_charger}
\end{equation*}
as claimed. The proof of~\eqref{eq_trigger_tilde} is completely analogous. 
\end{proof}

\subsection{Proof of Proposition \ref{prop:decay_v+-}}
\label{ss:interp}

Lemma~\ref{trigger} only bounds $p_H$, $\tilde{p}_H$ for $H = \useconstant{c:h_trigger}L_ {k}$.
We extend next these bounds to $H \geq 1$ with an interpolation argument, finishing the proof of Proposition~\ref{prop:decay_v+-}.

\begin{proof}[Proof of  Proposition \ref{prop:decay_v+-}]
We will only prove the inequality involving $v_{+}$ in \eqref{e:decay_v+-}, as the other inequality is analogous.
Given $\varepsilon > 0$, let $v=v_{+}+\varepsilon$, $v^{\prime}=v_+ + \varepsilon/2$ and fix $\useconstant{c:h_trigger}= \useconstant{c:h_trigger}(v^\prime)$, $\useconstantk{k:after_turbo_charger}=\useconstantk{k:after_turbo_charger}(v^\prime)$ as in Lemma \ref{trigger}. 
For $H \geq 1$, we define $\overline{k}$ as the unique integer such that
\begin{equation}
\label{e:interpolation_H}
\useconstant{c:h_trigger}L_{\overline{k}+1}\leq H < \useconstant{c:h_trigger} L_{\overline{k}+2}.
\end{equation}
Assume initially that $H$ is so large that
\begin{equation}\label{condition_to_H_large}
\overline{k}\geq \useconstantk{k:after_turbo_charger}, \quad \frac{1}{\ell_{\overline{k}}} \leq \frac{v - v^{\prime}}{3\lambda} 
\quad \text{ and } \quad \useconstant{c:h_trigger} L_{\overline{k}+2} \leq L_{\overline{k}+2}^{1+\nu}.
\end{equation}
The first condition together with Lemma~\ref{trigger} gives
\begin{equation}\label{eq_trigger_2}
p_{\useconstant{c:h_trigger}L_{\overline{k}}}(v^\prime)\leq e^{-4\kappa \log^{\gamma} L_{\overline{k}}}.
\end{equation}
To keep $X^y$, $y \in I_H(w)$ inside $B_{H}(w)$ for time $H$, define the set, for $w=(x,t)$,
\begin{equation*}
\cC_{H}(w)= \left\{ (z, s) \in B_H(w) \colon\, z \in \Z, s-t \in  \useconstant{c:h_trigger} L_{\overline{k}} \N_0 \right\}
\end{equation*}
as well as the event
\begin{equation*}
\hat{D}_{H}(w)= \bigcap_{y \in \cC_{H}(w)} \Big\{\sup_{s \in [0,\useconstant{c:h_trigger}L_{\overline{k}}]}| X^y_{s}-\pi_1(y)|\leq 4 \lambda \useconstant{c:h_trigger} L_{\overline{k}}\Big\},
\end{equation*}
which is similar to $\hat{D}_m$ in \eqref{e:rw_inside_Bm}. A computation as for \eqref{event_one} yields
\begin{equation}\label{Eq:inter_large_deviation}
\PP\big[ \hat{D}_{H}(w)^{c} \big]\leq 18\lambda \useconstant{c:h_trigger} L_{\overline{k}}^{2(1+\nu)^{2}}e^{-2\lambda \useconstant{c:h_trigger}L_{\overline{k}}},
\end{equation}
where we also used \eqref{e:interpolation_H} and $L_k \leq L_{k+1} \leq L_k^{1+\nu}$ for every $k$.

Define now $H^{\prime}=\lfloor H / \useconstant{c:h_trigger} L_{\overline{k}} \rfloor \useconstant{c:h_trigger} L_{\overline{k}}$ 
and two events
\begin{equation}\label{interpolation_event_1}
\mathcal{E}_1 := \big\{\exists \, y \in I_{H}(w)\cap \mathbb{L} : X^{y}_{H^\prime} - \pi_1(y) \geq v^{\prime} H \big\},
\end{equation}
\begin{equation}
\label{interpolation_event_2}
\mathcal{E}_2 := \big\{\exists \, y^{\prime} \in \left( w+[-4\lambda H, 5\lambda H) \times \{H^\prime\} \right) \cap \mathbb{L} : X^{y^{\prime}}_{H-H^\prime} - \pi_1(y^\prime) \geq \left(v-v^{\prime}\right)H \big\}.
\end{equation}
Note that $A_{H,w}(v)\cap \hat{D}_{H}(w) \subset \mathcal{E}_1 \cup \mathcal{E}_2$.
To bound the probability of $\mathcal{E}_1$, fix $w \in \R \times \R_+$ and cover $B_{H}(w)$ with boxes $B_m$ with indices $m$ in the set
\begin{equation*}
\begin{split}
M = \Bigg\{ & \left(\useconstant{c:h_trigger}, \overline{k}, w+\left(x \useconstant{c:h_trigger}\lambda L_{\overline{k}},\text{ } y \useconstant{c:h_trigger} L_{\overline{k}}\right)\right)\colon\\ 
& -\left\lceil \frac{4H}{\useconstant{c:h_trigger} L_{\overline{k}}}\right\rceil \leq x \leq \left\lceil \frac{5H}{\useconstant{c:h_trigger} L_{\overline{k}}}\right\rceil, \,\, 0\leq y \leq \left\lceil \frac{H}{\useconstant{c:h_trigger} L_{\overline{k}}}\right\rceil, \,\, x, y \in \mathbb{Z} \Bigg\}.
\end{split}
\end{equation*}
It is straightforward to verify that
$|M|\leq 36 L_{\overline{k}}^{2(1+\nu)^{2}}$.
Now note that, on the event $\bigcap_{m \in M}\left( A_{m}(v^\prime)\right)^{c}$, for any $y \in I_{H}(w)\cap \mathbb{L}$ we have
\begin{equation*}
\begin{split}
X^{y}_{H^\prime}-\pi_{1}(y) & = \sum_{j=0}^{\lfloor H / \useconstant{c:h_trigger} L_{\overline{k}} \rfloor  -1}\Big[X^{Y^{y}_{j \useconstant{c:h_trigger} L_{\overline{k}}}}_{\useconstant{c:h_trigger} L_{\overline{k}}}-X^{y}_{j \useconstant{c:h_trigger} L_{\overline{k}}}\Big] < \lfloor H / \useconstant{c:h_trigger} L_{\overline{k}} \rfloor v^{\prime}\useconstant{c:h_trigger} L_{\overline{k}} = v^{\prime}H^{\prime}\leq v^{\prime}H,
\end{split}
\end{equation*}
so that, by \eqref{eq_trigger_2},
\begin{equation*}
\PP[\mathcal{E}_1] \leq \PP \left[\cup_{m \in M} A_{m}(v^\prime) \right] \leq 36 L_{\overline{k}}^{2(1+\nu)^{2}} e^{-4\kappa \log^{\gamma} L_{\overline{k}}}.
\end{equation*}
To bound the probability of $\mathcal{E}_2$, 
note that, by \eqref{e:interpolation_H} and \eqref{condition_to_H_large}, 
\begin{equation*}
\useconstant{c:h_trigger} L_{\overline{k}} = \useconstant{c:h_trigger} \frac{L_{\overline{k}+1}}{\ell_{\overline{k}}} \leq \frac{H}{\ell_{\overline{k}}} \leq \frac{v - v^{\prime}}{3\lambda} H.
\end{equation*}
Since $H - H^\prime \leq \useconstant{c:h_trigger} L_{\overline{k}} $, Lemmas~\ref{l:bdincr} and \ref{l:PoiConc} imply
\begin{equation*}
\PP\left[X_{H-H^\prime} \geq (v-v^{\prime})H\right] \leq \PP\left[ X_{H-H^\prime} \geq 3\lambda \useconstant{c:h_trigger} L_{\overline{k}} \right] \leq e^{-  \lambda\useconstant{c:h_trigger} L_{\overline{k}}},
\end{equation*}
which together with a union bound and translation invariance leads to
\begin{equation*}
\PP[\mathcal{E}_2] \leq 9 \lambda H e^{ - \useconstant{c:h_trigger} \lambda L_{\overline{k}} } \leq 9 \useconstant{c:h_trigger} \lambda L_{\overline{k}}^{(1+\nu)^2} e^{ - \useconstant{c:h_trigger} \lambda L_{\overline{k}} }.
\end{equation*}
Gathering all these bounds and recalling $A_{H,w}(v)\cap \hat{D}_{H}(w) \subset \mathcal{E}_1 \cup \mathcal{E}_2$ we obtain
\begin{equation*}
\begin{split}
\PP & \big[A_{H,w}(v)\big] \leq \PP[\hat{D}_H^c] + \PP[\mathcal{E}_1]+ \PP[\mathcal{E}_2] \\
& \leq 
18\lambda \useconstant{c:h_trigger} L_{\overline{k}}^{2(1+\nu)^{2}}e^{-2\lambda \useconstant{c:h_trigger}L_{\overline{k}}}
+ 36 L_{\overline{k}}^{2(1+\nu)^{2}}e^{-4 \kappa \log^{\gamma} L_{\overline{k}}} + 9 \useconstant{c:h_trigger} \lambda L_{\overline{k}}^{(1+\nu)^2} e^{ - \useconstant{c:h_trigger} \lambda L_{\overline{k}} } \\
& \leq \useconstant{c:c_accommodate}e^{-\tfrac72 \kappa \log^{\gamma} L_{\overline{k}}} 
\leq \useconstant{c:c_accommodate}e^{- 3\kappa \log^{\gamma} H}
\end{split}
\end{equation*}
for some positive constant $\useconstant{c:c_accommodate}$, 
where for the last inequality we use \eqref{e:interpolation_H}, \eqref{condition_to_H_large} and \eqref{e:choice_nu} to obtain
$6(\log H)^\gamma \leq 6(1+\nu)^{3\gamma} (\log L_{k})^\gamma \leq 7 (\log L_{k})^\gamma$.

To conclude, take the supremum over $w$ and increase the constant $\useconstant{c:c_accommodate}$ if necessary to accommodate for smaller values of $H$.
\end{proof}

To finish the section, we show how Proposition~\ref{prop:decay_v+-}
implies that $v_{-} \leq v_{+}$.
The argument is contained in \cite[Corollary~3.3]{HKA19} but is simpler in our setting.
\begin{proof}[Proof of  Corollary \ref{cv+mv-}]
First note that, if $v_ {1} < v_{2}$ and $H>0$,
$p_{H}(v_{1}) + \tilde{p}_{H}(v_{2}) \geq 1$.
Assume by contradiction that $v_{-}>v_{+}$ and let $\varepsilon = \tfrac{1}{4}(v_{-} - v_{+})$. 
By Proposition~\ref{prop:decay_v+-}, $\lim_{H \to \infty} p_{H}(v_{+}+\varepsilon) + \tilde{p}_{H}(v_{-}-\varepsilon) = 0$,
but this is impossible since $v_+ + \varepsilon < v_- - \varepsilon$.
\end{proof}

\section{Proof of Proposition \ref{v+=v-}}
\label{s:v-=v+}

To prove Proposition~\ref{v+=v-}, we proceed by contradiction.
Assume $v_-\neq v_+$. By Corollary~\ref{cv+mv-}, this means $v_{-}<v_{+}$,
so we can define
\begin{equation}\label{theta}
\theta = \frac{v_{+}-v_{-}}{6} \in \big(0, \lambda / 6 \big)
\end{equation}
by Proposition~\ref{p:bounds_to_vs}. We follow the argument from \cite{blondel2020random, HKA19}, which is divided in two parts.
In Section~\ref{ss:trap_threatened}, we define trapped and threatened points, 
which are loosely speaking space-time locations that introduce delays
in the random walk path. We also show that any given point is threatened with high probability.
In Section~\ref{ss:proof_sharpness}, we use the existence of many threatened points to show $\liminf p_H(v)=0$ for a $v$ strictly less than $v_{+}$.
This provides a contradiction with the definition of $v_+$, implying $v_-=v_+$.

\subsection{Traps and threatened points}
\label{ss:trap_threatened}

We start this section with the following definition:

\begin{definition}
\label{d:trapped}
  Given $K \geq 1$ and $\theta$ as in \eqref{theta}, we say that a point $w \in \R \times \R_+$ is 
  \emph{$K$-trapped} if there exists some $y \in \big(w + [\theta K, 2 \theta K] \times \{0\} \big) \cap \mathbb{L}$ such that
  \begin{equation}\label{trappedpoint}
  X^y_K - \pi_1(y) \leq (v_- + \theta)K.
  \end{equation}
\end{definition}
\noindent
Note that Definition~\ref{d:trapped} applies to points $w \in \mathbb{R} \times \mathbb{R}_+$ not necessarily in $\mathbb{L}$. See Figure~\ref{fig:trapped_point} for an illustration of a $K$-trapped point.

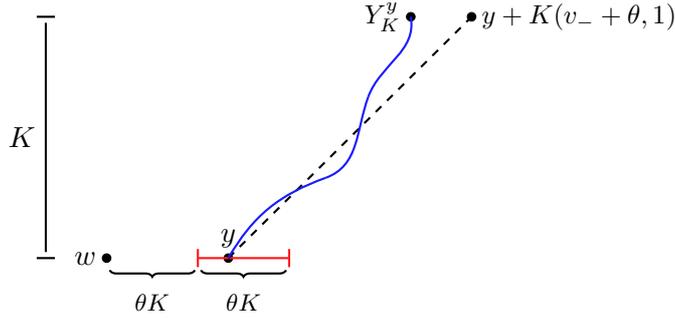
\begin{figure}
\begin{center}
\begin{tikzpicture}[scale=0.8]


\node[left] at (1,0) {$w$};
\node[above] at (3,0) {$y$};
\node[left] at (6,4) {\small{$Y^{y}_{K}$}};
\node[right] at (7,4) {\small{$y+K(v_{-}+\theta, 1)$}};
\node[left] at (0,2) {$K$};

\filldraw (1,0) circle (2pt);
\filldraw (3,0) circle (2pt);
\filldraw (6,4) circle (2pt);
\filldraw (7,4) circle (2pt);

\draw[red!90!][-, thick]  (2.5,0) -- (4,0);
\draw[-, thick]  (0,0.1) -- (0,3.9);
\draw[-, thick]  (-0.15,0) -- (0.15,0);
\draw[-, thick]  (-0.15,4) -- (0.15,4);
\draw[dashed, thick]  (3,0) -- (7,4);
\draw[red!90!][-, thick]  (2.5, -0.15) -- (2.5, 0.15);
\draw[red!90!][-, thick]  (4,-0.15) -- (4,0.15);
\draw [blue!90!] [-,thick] (3,0) to[out=60, in=200] (4.6,1.33);
\draw [blue!90!] [-,thick] (4.6,1.33) to[out=20, in=250] (5.3,2.67);
\draw [blue!90!] [-,thick] (5.3,2.67) to[out=70, in=280] (6,4);
\draw [thick,decoration={brace,mirror},decorate] (2.55,-.2) -- (3.95,-.2) node[midway,below,yshift=-.17cm] {\footnotesize{$\theta K$}};
\draw [thick,decoration={brace,mirror},decorate] (1.05,-.2) -- (2.45,-.2) node[midway,below,yshift=-.17cm] {\footnotesize{$\theta K$}};

\end{tikzpicture}
\caption{\small{An illustration of a $K$-trapped point $w$.}}
\label{fig:trapped_point}
\end{center}
\end{figure}

If $w$ is $K$-trapped, then any random walk $X^{y'}$ starting at a point $y' \in \LL$ near and to the right of $w$ will suffer a delay in its average speed in relation to $v_+$ after time $K$. Indeed, by monotonicity, for every $y' \in \big(w + [0,  \theta K] \times \{0\} \big) \cap \mathbb{L}$,
\begin{equation*}
  X^{y'}_K - \pi_1(y') \leq X^y_K - \pi_1(y) + 2 \theta K \leq (v_- + 3 \theta) K = (v_+ - 3\theta)K,
\end{equation*}
where $y$ is a point as in \eqref{trappedpoint} so that $\pi_1(y)-\pi_1(y')\leq 2\theta K$.

The first step is to show that a point is $K$-trapped with uniformly positive probability. 

\newconstant{c:sometrapped}
\newconstant{c:Hlower}

\begin{lemma}
  \label{sometrapped}
  There exist constants $\useconstant{c:sometrapped} > 0$, $\useconstant{c:Hlower}> 4/\theta + 1$ such that, for any $K \geq \useconstant{c:Hlower}$,
  \begin{equation}
   \label{eqsometrapped}
 \inf_{w \in \R \times \R_+ } \PP \big[ \text{$w$ is $K$-trapped} \big] \geq \useconstant{c:sometrapped}.
  \end{equation}
\end{lemma}
\begin{proof}
We follow \cite[Lemma 5.2]{blondel2020random}.
First note that, since $\theta>0$,
\[
\useconstant{c:sometrapped} := \frac{1}{2} \left\lceil \frac{2}{\theta} \right\rceil^{-1} \liminf_{K \to \infty} \tilde{p}_K(v_-+\theta) > 0
\]
by the definition of $v_-$. Thus there exists $\useconstant{c:Hlower}> 4/\theta + 1$ such that
\[
\left\lceil \frac{2}{\theta} \right\rceil^{-1} \inf_{K \geq \useconstant{c:Hlower}} \tilde{p}_K(v_-+\theta) \geq \useconstant{c:sometrapped}.
\]
Now, if $K \geq \useconstant{c:Hlower}$,
\[
\begin{aligned}
\useconstant{c:sometrapped} 
& \leq \left\lceil \frac{2}{\theta} \right\rceil^{-1}
\sup_{w \in [0,1) \times \{0\}} 
\PP \left( 
\begin{array}{l}
\text{there exists } y \in (w + [0,K)\times \{0\})\cap \LL\\
\text{ such that } X^y_K - \pi_1(y) \leq (v_- + \theta) K
\end{array}
 \right) \\
 & \leq 
\sup_{w \in [0,1) \times \{0\}} 
\PP \left( 
\begin{array}{l}
\text{there exists } y \in (w + [0, \theta K/2)\times \{0\})\cap \LL\\
\text{ such that } X^y_K - \pi_1(y) \leq (v_- + \theta) K
\end{array}
 \right) \\
 & \leq 
\inf_{w \in [0,1) \times \{0\}} 
\PP \left( 
\begin{array}{l}
\text{there exists } y \in (w + [0, \theta K)\times \{0\})\cap \LL\\
\text{ such that } X^y_K - \pi_1(y) \leq (v_- + \theta) K
\end{array}
 \right) \\
& = 
\inf_{w \in [0,1) \times \{0\}} 
\PP \left( 
\begin{array}{l}
\text{there exists } y \in (w + [\theta K, 2\theta K)\times \{0\})\cap \LL\\
\text{ such that } X^y_K - \pi_1(y) \leq (v_- + \theta) K
\end{array}
 \right),
\end{aligned}
\]
where the steps are justified as follows: For the second inequality, 
divide the interval $[0,K)$ into segments of length $\theta K /2$, apply
the union bound and translation invariance.
For the third inequality, observe that, 
since $\theta K \geq \theta (4/\theta+1)>4$, for any $w \in [0,1) \times \{0\}$ the interval
$w + [0,\theta K/2) \times \{0\}$ is contained in every interval of the form $w'+[0,\theta K)\times \{0\}$
with $w' \in [-1,0) \times \{0\}$, and then apply translation invariance.
The last equality is again a consequence of translation invariance.
\end{proof}

From \eqref{eqsometrapped} it is intuitive to expect that there is a density of trapped points in $\R^2$.
It is difficult however to exclude the possibility that random walk trajectories might avoid these traps or spend a  small fraction of time near them. 
To deal with this, we introduce next \emph{threatened} points, which are a weaker type of trap but much more likely.

\begin{definition}
For $K \geq 1$ and $r \in \N$, we say that a point $w \in \R \times \R_+$ is \emph{$(K, r)$-threatened} if $w + j K (v_+, 1)$ is $K$-trapped for some $j = 0, \dots, r-1$.
\end{definition}

In words, a space-time point is threatened if a certain line segment with slope $v_{+}$ starting from it contains at least one trapped point (see Figure \ref{fig:rwfast}).
The advantage of threatened points is that they occur with high probability when $r$ is large, as will be shown in Lemma~\ref{threatenedpoints} below. 
Moreover, if a random walk starts from such a point, it will either be delayed with respect to $v_+$ after time $rK$ or attain average speed larger that $v_{+}$ in a time interval of length $K$. This is the content of the next result.

\begin{lemma}
\label{rwfast}
  For any $r \in \N$ and any $K\geq \useconstant{c:Hlower}$, if $y \in \LL$ is $(K, r)$-threatened then either
\begin{equation}\label{speedup}
      X_{(j + 1)K}^y - X_{jK}^y
      \geq \Big( v_+ + \frac{\theta}{2 r} \Big)K \quad \text{ for some $j = 0, \dots, r - 1$,}
\end{equation}
 or
\begin{equation}\label{delay}
      X^y_{rK} - \pi_1(y)
      \leq \Big( v_+ - \frac{\theta}{2 r} \Big) r K.
\end{equation}
\end{lemma}
\begin{proof}
See \cite[Lemma 5.4]{blondel2020random}.
\end{proof}

\begin{figure}
\begin{center}
\begin{tikzpicture}[scale=0.8]


\node[left] at (0,0) {$rK$};
\node[left] at (3,-4) {\small{$y$}};
\node[left] at (3.5,-3.5) {\scriptsize{$y+K(v_{+}, 1)$}};
\node[below] at (8.7,0.7) {\scriptsize{$y^{\prime}$}};
\node[right] at (11,4) {\scriptsize{$y+rK(v_{+}, 1)$}};
\node[left] at (9.9,4.1) {\small{$Y^{y}_{rK}$}};
\node[left] at (6.8,1.2) {\scriptsize{$y+j_{0}K(v_{+}, 1)$}};

\filldraw (3,-4) circle (2pt);
\filldraw (8.7,0.7) circle (2pt);
\filldraw (7.7,0.7) circle (2pt);
\filldraw (11,4) circle (2pt);
\filldraw (10,4) circle (2pt);


\draw[-, thick]  (0,-3.9) -- (0,3.9);
\draw[-, thick]  (-0.15,-4) -- (0.15,-4);
\draw[-, thick]  (-0.15,4) -- (0.15,4);
\draw[dashed, thick]  (3,-4) -- (11,4);
\draw[->, thick]  (3,-4) -- (3.5,-3.5);
\draw [blue!90!] [-,thick] (3,-4) to[out=100, in=250] (4,-3);
\draw [blue!90!] [-,thick] (4,-3) to[out=70, in=200] (6,-1);
\draw [blue!90!] [-,thick] (6,-1) to[out=20, in=260] (8,1);
\draw [blue!90!] [-,thick] (8,1) to[out=70, in=225] (10,4);
\draw [red!90!][-, thick]  (8.3,0.7) -- (9,0.7);
\draw [red!90!][-, thick]  (8.3,0.6) -- (8.3,0.8);
\draw [red!90!][-, thick]  (9,0.6) -- (9,0.8);
\draw [blue!90!] [-,thick] (8.7,0.7) to[out=80, in=245] (8.6,2);
\draw [->,thick] (7.7,0.7) to[out=125, in=345] (7,1.2);

\end{tikzpicture}
\caption{\small{Illustration for Lemma \ref{rwfast}. A random walk starting on a point that is $(K, r)$-threatened is likely to experience a delay in its average speed relative to $v_{+}$ after time $rK$.}}
\label{fig:rwfast}
\end{center}
\end{figure}
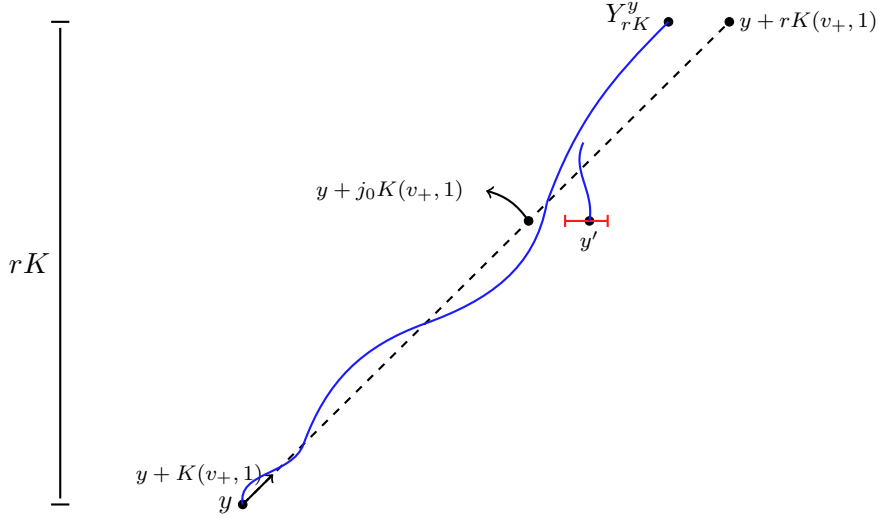

Next we show that points are threatened with overwhelming probability.
Recall the constants $\useconstant{c:sometrapped}, \useconstant{c:Hlower}$ from Lemma~\ref{sometrapped}.

\newconstant{c:highthreatened}
\newconstant{c:first_n_to_highthreatened}
\newconstant{c:final_threat_const}

\begin{lemma}\label{threatenedpoints}
There exists a constant $\useconstant{c:highthreatened}>0$ such that, for any $r \in \N$ and $K\geq \useconstant{c:Hlower}$, 
\begin{equation*}
\sup_{w\in \R \times \R_+}  \PP \big[ \text{w is not $(K, r)$-threatened} \big] \leq \useconstant{c:highthreatened} r^{-1000}.
\end{equation*}
\end{lemma}
\begin{proof}
Recall the constants $v_\circ, \gamma_\circ, \usebigconstant{c:decoupling_2}, \usebigconstant{c:decoupling_3}$ in \eqref{e:DEC}
and $v_\star$ in \eqref{e:BAL}.
Since we are assuming $v_\star>v_\circ$, 
we may fix an integer $L\geq 3$ such that
\begin{equation}\label{e:Lthreat}
(L-2)(v_\star-v_\circ) \geq \usebigconstant{c:decoupling_2} + 1.
\end{equation}
Fix also $\gamma_\bullet \in (1,\gamma_\circ \wedge 2)$ and an integer $n_0 \geq 3$ such that
\begin{equation}\label{n_zero_high_threat}
2n^{\gamma_\bullet} - (n+1)^{\gamma_\bullet} \geq 1 \quad \text{ and } \quad L^{n} - 4n^2 - 8\lambda \geq 1 \vee \usebigconstant{c:decoupling_3} \;\; \forall\, n\geq n_{0}.
\end{equation}
Let us first prove the statement for $r$ of the form $r = L^n$, $ n\geq n_0$. In order to do so, for each $n\geq n_{0}$ and $K\geq 1$, 
consider the quantity
\begin{equation}\label{quantity_threat}
q^{K}_{n}= \sup_{w \in \R \times \R_+}\PP \big[ \text{$w$ is not $(K, L^n)$-threatened} \big]. 
\end{equation}
Given $w = (x,t) \in \R \times \R_+$, define the set
\begin{equation*}
\begin{split}
\cC^{K}_{n}=
\bigg\{ (z,s) \in \LL \colon\, 
&-n^2 - 4 \lambda K \leq z-x \leq L^{n_1} K v_+ +4\lambda K + n^2,\\
& s-t \in K \{0, \ldots, L^{n+1}-1 \}\bigg\}
\end{split}
\end{equation*}
and the event
\begin{equation*}
D^{K}_{n} = \bigcap_{z \in \cC^{K}_{n}}\Big\{\sup_{s \in [0,K]}|X_{s}^{z}-\pi(z)|\leq 3\lambda K + n^2\Big\}.
\end{equation*}
To bound the probability of $D^{K}_{n}$, first use $v_+ \leq \lambda$, $\lambda, K \geq 1$ and \eqref{n_zero_high_threat} to bound
\begin{equation*}
|\cC^{K}_{n}|\leq L^{n+1}\left( L^{n+1}Kv_+ + 2n^2 + 8\lambda K \right)\leq 2L^2 \lambda K L^{2n}.
\end{equation*}
By Lemma~\ref{l:PoiConc}, a union bound, translation invariance and $t e^{-t} \leq 1/2$,
\begin{equation}\label{second_error_threatened}
\PP [\left(D^{K}_{n}\right)^c]\leq 2L^2 \lambda K L^{2n} e^{-\lambda K} e^{-n^2} \leq L^{2n+2} e^{-n^2}.
\end{equation}
Define now $w_L = y + (L-1) L^nK(v_{+},1)$ and
\begin{equation*}
A_1= \{ \text{$w$ is not $(K, L^n)$-threatened}\}, \qquad A_2= \{ \text{$w_L$ is not $(K, L^n)$-threatened}\}.
\end{equation*}
Observe that 
\begin{equation}\label{inclusion_events_high_threatened}
\{\text{$w$ is not $(K, L^{n+1})$-threatened}\} \subset A_1\cap A_2.
\end{equation}
Furthermore, on $D^{K}_{n}$, occurrence of $A_{1}$, $A_{2}$ is determined by 
$\eta, \Pi$ inside boxes $B_1$, $B_2$ respectively, where
\begin{equation*}
B_1= w + [-n^2 - 4\lambda K, L^n K v_{+}+ n^2 + 4\lambda K]\times [0, L^n K],
\end{equation*}
\begin{equation*}
B_2= w_L + [-n^2 - 4\lambda K, L^n K v_{+}+ n^2 + 4\lambda K]\times [0, L^n K].
\end{equation*}
Denote by $\dist_{H}$, $\dist_{V}$, $s$ the horizontal/vertical distances and height of $B_1$, $B_2$.
Note that $\dist_{V}=(L-2)L^n K$ and $s= L^n K$. 
Using $v_{+} \geq v_{\star}$ and \eqref{e:Lthreat}--\eqref{n_zero_high_threat}, we obtain
\begin{equation}\label{condition_decoupling_high_threat}
\begin{split}
\dist_H &= (L - 2) L^{n} K v_{+} - 4n^2-8\lambda K\\
&\geq v_\circ \dist_V + (L-2)(v_\star - v_\circ) L^n K  - 4n^2-8\lambda K \geq v_\circ \dist_V + \usebigconstant{c:decoupling_2} s + \usebigconstant{c:decoupling_3}. 
\end{split}
\end{equation}
Thus we may apply \eqref{e:DEC} (and Remark~\ref{r:ext_decouple}), yielding
\begin{equation}\label{Aply_dec_threat}
\begin{split}
\PP \left[A_{1}\cap A_{2}\right]&= \PP \left[A_{1}\cap A_{2} \cap D^{K}_{n}\right]+ \PP \left[A_{1}\cap A_{2} \cap \left(D^{K}_{n}\right)^c\right]\\ 
&\leq \PP \left[A_{1}\right] \PP \left[A_{2}\right] + C_\circ e^{-\kappa_\circ\log^{\gamma_\circ}\dist_H} + 3 \PP \left[\left(D^{K}_{n}\right)^c\right]\\
&\leq \PP \left[A_{1}\right] \PP \left[A_{2}\right] + C_\circ e^{-\kappa_\circ\log^{\gamma_\circ}L^n} + 3L^{2n+2}e^{-n^2}
\end{split}
\end{equation}
where we used \eqref{second_error_threatened} and \eqref{e:Lthreat}--\eqref{n_zero_high_threat} again to obtain $d_H \geq L^n K \geq L^n$.
Hence,
\begin{equation}\label{relation_scales_threat}
q_{n+1}^K \leq \left(q_{n}^K\right)^2 + C_\circ e^{-\kappa_\circ\log^{\gamma_\circ} L^n} + 3L^{2n+2}e^{-n^2}.
\end{equation}
Recall the constant $\useconstant{c:sometrapped}>0$ from Lemma~\ref{sometrapped} and fix $0<a<1$ such that $a^{n_0^{\gamma_\bullet}} \geq (1- \useconstant{c:sometrapped})$. 
Fix also a constant $\useconstant{c:first_n_to_highthreatened}$ such that, for every $n \geq n_0$,
\begin{equation}\label{small_error_threat_1}
C_\circ \exp\left\{-\kappa_\circ(\log L^{\useconstant{c:first_n_to_highthreatened} + n})^{\gamma_\circ} -(n+1)^{\gamma_\bullet}\log a\right\}\leq \frac{1-a}{2}
\end{equation}
and
\begin{equation}\label{small_error_threat_2}
3 L^2 \exp\left\{2(\useconstant{c:first_n_to_highthreatened} + n)\log L - (\useconstant{c:first_n_to_highthreatened} +n)^2 -(n+1)^{\gamma_\bullet}\log a \right\}\leq \frac{1-a}{2}.
\end{equation}

We claim that if $K\geq \useconstant{c:Hlower}$ then
\begin{equation}
\label{e:claimthreat}
q^{K}_{\useconstant{c:first_n_to_highthreatened} + n}\leq a^{n^{\gamma_\bullet}} \quad \text{for all } n \geq n_{0}.
\end{equation}
Let us prove \eqref{e:claimthreat} by induction. 
Indeed, for $n=n_0$ and any $w \in \R \times \R_+$,
\begin{equation*}
\begin{split}
\PP \big[ \text{$w$ is not $(K, L^{\useconstant{c:first_n_to_highthreatened} + n_{0}})$-threatened} \big] 
&\leq \PP  \big[ \text{$w$ is not $K$-trapped} \big]\leq 1- \useconstant{c:sometrapped}
\leq a^{n_0^{\gamma_\bullet}}
\end{split}
\end{equation*}
by Lemma~\ref{sometrapped} and the definition of $a$.
Assume now that $q^{K}_{\useconstant{c:first_n_to_highthreatened} + n}\leq a^{n^{\gamma_\bullet}}$ for some $n\geq n_0$.
By \eqref{relation_scales_threat}, \eqref{small_error_threat_1}, \eqref{small_error_threat_2} and \eqref{n_zero_high_threat},
\begin{equation*}
\begin{split}
\frac{q_{\useconstant{c:first_n_to_highthreatened} + n + 1}^K}{a^{(n+1)^{\gamma_\bullet}}} 
& \leq a^{2n^{\gamma_\bullet} - (n+1)^{\gamma_\bullet}} + \frac{1-a}{2} + \frac{1-a}{2} \leq a + 1-a = 1,
\end{split}
\end{equation*}
finishing the proof of \eqref{e:claimthreat}.
Fix now $\useconstant{c:final_threat_const} \geq n_{0}$ such that $a^{\useconstant{c:final_threat_const}}\leq L^{-1000}$ and so
\begin{equation*}
q^{K}_{ \useconstant{c:first_n_to_highthreatened} + n}\leq (a^{n^{\gamma_\bullet-1}})^n \leq L^{-1000n} \qquad \text{for all } n\geq n_1 := \useconstant{c:final_threat_const}^{(\gamma_\bullet-1)^{-1}}.
\end{equation*}
For $r \geq L^{\useconstant{c:first_n_to_highthreatened} + n_1}$, let $\bar{n}$ be the unique integer such that $L^{\bar{n}}\leq r < L^{\bar{n}+1}$. 
Then
\begin{equation*}
\begin{split}
\sup_{w \in \R \times \R_+}\PP \big[ \text{$w$ is not $(K, r)$-threatened} \big]
& \leq q^{K}_{\useconstant{c:first_n_to_highthreatened} +(\bar{n} - \useconstant{c:first_n_to_highthreatened})}
\leq L^{1000(1+ \useconstant{c:first_n_to_highthreatened})}r^{-1000}.
\end{split}
\end{equation*}
To conclude, increase $\useconstant{c:highthreatened}$ if needed to accommodate smaller values of $r$.
\end{proof}

\subsection{Proof of Proposition \ref{v+=v-}}
\label{ss:proof_sharpness}

A first sign of trouble with the assumption $v_+ > v_-$ can be seen next.
\begin{lemma}\label{firstdelay}
For any $\varepsilon>0$, there exist $r=r(\varepsilon) \in \N$ and $H_0=H_0(\varepsilon) > 0$ such that
\begin{equation*}
p_H \Big(v_+ - \frac{\theta}{4r} \Big) \leq \varepsilon \quad \text{ for all } H \geq H_0.
\end{equation*}
\end{lemma}
\begin{proof}
Recall the constant $\useconstant{c:Hlower}$ from Lemmas~\ref{sometrapped} and \ref{threatenedpoints} and let $w \in \R \times \R_+$, $r\geq 1$, and $H\geq r \useconstant{c:Hlower}$. Define a sequence of elements of $\LL$ by
\begin{equation*}
y_{0}=(\lceil \pi_{1}(w)\rceil, \pi_{2}(w)), \quad y_i= y_{0}+i(\lfloor\theta H/4r \rfloor,0), \quad i \in \{0,\cdots,m\},
\end{equation*}
where $m$ is the first index such that $ y_{m} \notin \left[w, w+ \lambda H\right]$ (see Figure \ref{fig:points_yi}).
Since $\useconstant{c:Hlower} > 4/\theta$,
\begin{equation*}
\frac{\theta H}{4r}\geq \frac{\theta}{4}\useconstant{c:Hlower}>1
\end{equation*}
so that the sequence above is not constant.
Define next the events
\begin{align*}
E_1(H,r,w):=&\big\{  \exists \, i\in\left\{0,\dots, m \right\}:\ y_i\text{ is not }\big( \tfrac{H}{r}, r \big)\text{-threatened}\big\},\\
E_2(H,r,w):=&\bigcup_{\substack{0 \leq i \leq m \\ 0 \leq j \leq r-1}} \Big\{X_{(j + 1)\frac{H}{r}}^{y_i} - X_{j\frac{H}{r}}^{y_i} \geq \Big( v_+ + \frac{\theta}{2 r} \Big)\frac{H}{r}\Big\}.
\end{align*}
By Lemma~\ref{rwfast}, on $[E_1(H,r,w)\cup E_2(H,r,w)]^c$ it holds that, for all $ 0 \leq i \leq m$,
\begin{equation}
\label{e:delayyi}
X^{y_{i}}_H - \pi_1(y_{i}) = X^{y_{i}}_{r\frac{H}{r}} - \pi_1(y_{i}) \leq  \Big(v_+ - \frac{\theta}{2r}\Big)r\frac{H}{r}
                 =  \Big(v_+ - \frac{\theta}{2r}\Big)H.                
\end{equation}
Now note that, for any $y \in [w, w+\lambda H] \cap \mathbb{L}$, there exists $y_{i}$ such that $0\leq \pi_{1}(y_{i})-\pi_{1}(y)\leq \theta H/4r$.
Together with \eqref{e:delayyi} and monotonicity, this implies 
\begin{equation*}
X^{y}_{H} - \pi_{1}(y)\leq X^{y_i}_{H} - \pi_{1}(y_i) + \frac{\theta H}{4r}\leq \Big(v_+ - \frac{\theta}{4r}\Big)H
\end{equation*} 
on $[E_1(H,r,w)\cup E_2(H,r,w)]^c$. It follows that
\begin{equation*}
\PP \Big[ A_{H,w}\Big(v_+-\frac{\theta}{4r}\Big)\Big]\leq \PP \left[ E_1(H,r,w)\right]+\PP \left[ E_2(H,r,w)\right].
\end{equation*}

\begin{figure}
\centering
\begin{tikzpicture}[scale=0.8]

\node[below] at (0.5,0) {$y_{0}$};
\filldraw [blue!90!]    (0.5,0)        circle (1.5pt);
\node[below] at (4,0) {$\cdots$};
\node[below] at (5.5,0) {$y_{i-1}$};
\filldraw [blue!90!]    (5.5,0)        circle (1.5pt);
\node[below] at (6.5,0) {$y$};
\filldraw [purple!90!]    (6.5,0)        circle (1.5pt);
\node[below] at (7.5,0) {$y_{i}$};
\filldraw [blue!90!]    (7.5,0)        circle (1.5pt);
\node[below] at (9,0) {$\cdots$};
\node[above] at (0,0) {${w}$};
\node[left] at (12.3,0.3) {\small{$w+\lambda H$}};
\filldraw [black!90!]    (0,0)        circle (2pt);
\filldraw [black!90!]    (12,0)        circle (2pt);
\node[below] at (12.5,0) {$y_{m}$};
\filldraw [blue!90!]    (12.5,0)        circle (1.5pt);


\draw[-, thick]  (-0.5,0) -- (13.5,0);

\end{tikzpicture}
\caption{\small{The sequence of points $y_{i} \in \mathbb{L}$ in the interval $[w, w+\lambda H]$.}}
\label{fig:points_yi}
\end{figure}
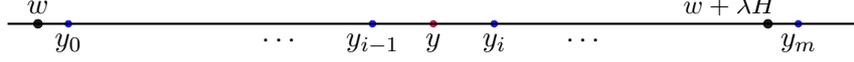

To bound $\PP \left[ E_1(H,r,w)\right]$, note that $m+1 \leq 9r\lambda / \theta$ so that
\begin{equation*}
\PP \left[ E_1(H,r,w)\right] \leq   \frac{9\lambda}{\theta}\useconstant{c:highthreatened} r^{-999}
\end{equation*}
by Lemma~\ref{threatenedpoints}.
Thus, given $\varepsilon>0$, we can fix $r=r(\varepsilon)$ so that
\begin{equation*}
\PP \left[E_1(H,r,w)\right] \leq \frac{\varepsilon}{2} \quad \text{ for all } w \in \R \times \R_+ \text{ and } H \geq r \useconstant{c:Hlower}.
\end{equation*}
On the other hand, Proposition \ref{prop:decay_v+-} gives
\begin{equation*}
\PP \left[E_2(H,r,w)\right]\leq \frac{9r^2\lambda}{\theta}\useconstant{c:c_accommodate} e^{-3 \kappa \log^{\gamma} \frac{H}{r}},
\end{equation*}
where $\useconstant{c:c_accommodate}=\useconstant{c:c_accommodate}(\theta/2r)$.
Then, for $r$ fixed as above, we can choose $H_0= H_{0}(\varepsilon)$ large enough such that
$\PP \left[E_2(H,r,w)\right] \leq \frac{\varepsilon}{2}$ for all $H \geq H_0$,
and this concludes the proof.
\end{proof}

Note that $r(\varepsilon)$ provided in the proof above tends to infinity as $\varepsilon \downarrow 0$,
so Lemma~\ref{firstdelay} by itself is not enough to provide a contradiction with the definition of $v_+$.
To go around this issue, we use Lemma~\ref{firstdelay} as input for a renormalization argument
similar to that of Section~\ref{ss:interp}, as described next.

Consider the sequence of scales defined in~\eqref{def_scales_rec}.
Recall $2>(1+\nu)^\gamma$ as well as the constants $\useconstant{c:general_recursion}$, $\useconstantk{k:scale_to_boxes}$ from Lemma~\ref{l:recgen}.
Now fix an integer $k_\bullet \geq \useconstantk{k:scale_to_boxes}$ such that
\begin{equation}\label{conditions_by_delay}
\sum_{k \geq k_\bullet} \frac{1}{k^2} \leq  1, \quad \sqrt{\ell_k} \geq k^2  \quad \text{ and } \quad 3 \useconstant{c:general_recursion} \ell_k^4 e^{-(2-(1+\nu)^\gamma) \kappa (\log L_k)^\gamma} \leq 1 \quad \text{ for all } k \geq k_\bullet.
\end{equation}
By Lemma~\ref{firstdelay}, there exist $H_\bullet \geq 1$, $r_\bullet \in \N$ such that
\begin{equation}
\label{e:delay_basecase}
p_{h L_{k_\bullet}}\Big(v_+ - \frac{\theta}{4r_\bullet} \Big) \leq e^{-\kappa (\log L_{k_\bullet})^\gamma} 
\quad \text{ for all } h \geq H_\bullet / L_{k_\bullet}.
\end{equation}
Denote by $\useconstant{c:c_accommodate}=\useconstant{c:c_accommodate}(\varepsilon_\bullet)$ the constant from Proposition~\ref{prop:decay_v+-} corresponding to 
$\varepsilon_\bullet = \theta/(8 r_\bullet)$, and fix a further constant $c_\bullet \geq H_\bullet/L_{k_\bullet}$ such that
\begin{equation}\label{e:delay_trigger}
\useconstant{c:c_accommodate} e^{-3\kappa (\log h)^\gamma} \leq 1 \qquad \text{ for all } h \geq c_\bullet.
\end{equation}
Finally, we recursively define a sequence of speeds as
\[
v_{k_\bullet} = v_+ - 2\varepsilon_\bullet, \qquad v_{k+1} = v_k + \frac{\varepsilon_\bullet}{k^2},  \;\; k \geq k_\bullet.
\]
Note that $v_k$ is increasing with limit $v_\infty \leq v_+ - \varepsilon_\bullet$.

With these definitions in hand, we obtain the following absurdity.
 
\begin{lemma}\label{Propimptheo}
For all $h\geq c_\bullet$ and all $k \geq k_\bullet$,
$\displaystyle 
p_{h L_k}(v_+-\varepsilon_\bullet) \leq e^{-\kappa (\log L_k)^\gamma}.
$
\end{lemma}
\begin{proof}
Note first that, by Proposition~\ref{prop:decay_v+-}, if $h \geq c_\bullet$ then
\begin{equation*}
p_{h L_k}(v_+ - \varepsilon_\bullet) \leq \useconstant{c:c_accommodate} e^{-3 \kappa (\log h L_k)^\gamma}
\leq \useconstant{c:c_accommodate} e^{-3 \kappa (\log h)^\gamma} e^{-3\kappa (\log L_k)^\gamma}
\leq e^{-3\kappa (\log L_k)^\gamma}
\end{equation*}
by \eqref{e:delay_trigger}.
Now take $k \geq k_\bullet$ and set $v_{\max} = v_+ + \varepsilon_\bullet$, $v_{\min} = v_k$.
Note that
\[
\bar{v} = v_{\min} + \frac{v_{\max} - v_{\min}}{\sqrt{\ell_k}} \leq v_k + \frac{\varepsilon_\bullet}{\sqrt{\ell_k}} \leq v_{k+1}
\]
by \eqref{conditions_by_delay}. Monotonicity and Lemma~\ref{l:recgen} then yield
\begin{equation}
\label{e:recdelay}
p_{h L_{k+1}}(v_{k+1}) \leq \useconstant{c:general_recursion} \ell_k^4 \left\{ p_{h L_k}(v_k)^2 + 2 e^{-2 \kappa (\log L_k)^\gamma} \right\}
\quad \text{ for all } h \geq c_\bullet \text{ and } k \geq k_\bullet.
\end{equation}
Fix now $h \geq c_\bullet$ and let us prove by induction that
\begin{equation}
\label{e:delaysequence}
p_{h L_{k}}(v_{k}) \leq e^{- \kappa (\log L_k)^\gamma} \quad \text{ for all } k \geq k_\bullet.
\end{equation}
Indeed, the base case $k=k_\bullet$ is given by \eqref{e:delay_basecase}.
Assuming that \eqref{e:delaysequence} holds for some $k \geq k_\bullet$, we apply \eqref{e:recdelay} to obtain
\[
p_{h L_{k+1}}(v_{k+1}) \leq 3 \useconstant{c:general_recursion} \ell_k^4 e^{-2\kappa (\log L_k)^\gamma}
\]
so that
\[
e^{\kappa (\log L_{k+1})^\gamma} p_{h L_{k+1}}(v_{k+1}) \leq 3 \useconstant{c:general_recursion} \ell_k^4 e^{-(2-(1+\nu)^\gamma) \kappa (\log L_k)^\gamma} \leq 1
\]
by \eqref{conditions_by_delay}, where we used $L_{k+1} \leq L_k^{1+\nu}$.
To conclude, note that, since $v_k \leq v_+ - \varepsilon_\bullet$ for all $k \geq k_\bullet$, 
monotonicity and \eqref{e:delaysequence} together imply the statement.
\end{proof}

Proposition~\ref{v+=v-} now easily follows from Lemma~\ref{Propimptheo}.

\begin{proof}[Proof of Proposition \ref{v+=v-}]
If $v_+ \neq v_-$ then $\theta>0$. By Lemma~\ref{Propimptheo}, there exists $\varepsilon>0$ such that $\liminf_{H\to\infty}p_{H}(v_{+}-\varepsilon)=0$, but this contradicts the definition of $v_{+}$.
\end{proof}

\section{Proof of Theorem \ref{t:maingeneral}}
\label{s:proof_main}

Finally, using Propositions~\ref{p:bounds_to_vs},~\ref{prop:decay_v+-}, and~\ref{v+=v-} together with the Borel-Cantelli Lemma, we can finish the proof of Theorem~\ref{t:maingeneral}.

\begin{proof}[Proof of Theorem \ref{t:maingeneral}]
By Proposition~\ref{v+=v-}, we can define $v:=v_{-}= v_{+}$, and $v \geq v_{\star}$ follows from Proposition~\ref{p:bounds_to_vs}.
To show \eqref{e:LLNdecay}, fix $\varepsilon > 0$ and note that, for $t \ge 1$, 
\[
\PP\left[|X_t-tv|\geq \tfrac12 \varepsilon t \right] \leq p_t(v_+ + \tfrac12 \varepsilon) + p_t(v_- - \tfrac12 \varepsilon) \leq 2 \useconstant{c:c_accommodate} e^{-3\kappa (\log t)^\gamma}
\]
by Proposition~\ref{prop:decay_v+-},
where $\useconstant{c:c_accommodate}=\useconstant{c:c_accommodate}(\varepsilon/2)$.
Thus, for $L \in \N$,
the event
\[
\mathcal{U}_L = \left\{ \exists\, n \in \N, n \geq L \colon |X_n-nv| > \tfrac12 \varepsilon n \right\}
\]
has probability
\[
\begin{split}
\PP \left[\mathcal{U}_L\right]
\leq 2 \useconstant{c:c_accommodate} \sum_{n = L}^\infty e^{-3\kappa (\log n)^\gamma}
\leq K_1 e^{-2 \kappa (\log L)^\gamma}
\end{split}
\]
where $K_1 = K_1(\varepsilon) =  2 \useconstant{c:c_accommodate} \sum_{n\geq 1} e^{-\kappa (\log n)^\gamma}$.

To control the deviation between integer times, define, for each $L\in\N$,
\begin{equation*}
\mathcal{V}_L=\Big\{\exists n \geq L \colon\, \sup_{s \in [0,1]}\left\lvert X_{n+s}-X_{n} \right\rvert \geq 2 \sqrt{n} \Big\}.
\end{equation*}
Lemma~\ref{l:PoiConc} and a union bound imply
$\PP \left[ \mathcal{V}_L \right] \leq K_2 e^{- \sqrt{L}}$
where $K_2 = e^{2\lambda} \sum_{n=1}^\infty e^{-\sqrt{n}}$.

Assume now that $T \in [2, \infty)$ is large enough such that
\begin{equation}\label{e:condT}
\log(T-1)^\gamma \geq \tfrac12 (\log T)^\gamma, \;\; \sqrt{T-1} \geq \kappa (\log T)^\gamma \;\; \text{ and } \;\; 2\sqrt{t} + v < \tfrac12 \varepsilon t \; \forall\, t \geq T.
\end{equation}
Let $L = \lfloor T \rfloor$.
On the event $\mathcal{U}_L^c \cap \mathcal{V}_L^c$, we have, for any $t \geq T$,
\begin{equation*}
|X_t - t v| \leq |X_{\lfloor t \rfloor} - \lfloor t \rfloor v| +|X_t - X_{\lfloor t \rfloor}| + v
\leq \tfrac12 \varepsilon t + 2\sqrt{t} + v < \varepsilon t,
\end{equation*}
so that, setting $K=K_1 + K_2$, 
\begin{equation*}
\begin{split}
\PP\left[\exists t \geq T \colon |X_t - tv| \geq \varepsilon t \right] & \leq \PP\left[\mathcal{U}_L\right] + \PP\left[ \mathcal{V}_L\right] \\
& \leq K_1 e^{-2\kappa (\log L)^\gamma} + K_2 e^{-\sqrt{L}} \leq K e^{-\kappa (\log T)^\gamma}
\end{split}
\end{equation*}
where for the last inequality we used \eqref{e:condT} and $L \geq T-1$.
To conclude \eqref{e:LLNdecay}, we only need to increase $K$ if needed to account for smaller values of $T$.

Now, since \eqref{e:LLNdecay} is summable along integer $T$, the Borel-Cantelli lemma directly implies the almost sure convergence in \eqref{e:genLLN}.
To obtain convergence in $L^p$, observe that, by Lemma~\ref{l:bdincr}, $|\tfrac1t X_t - v|^p$ is uniformly integrable for each $p \geq 1$.
\end{proof}

\section{Proof of Proposition~\ref{p:DEC_ASEP}}
\label{s:asep}

In this section we prove the lateral decoupling for the asymmetric exclusion process, Proposition~\ref{p:DEC_ASEP}. 
The proof will in fact show exponential decay in $d_H$ in the right-hand side of \eqref{e:DEC}.
We start by providing a particular construction of the process and introducing higher-class particles. This concept will be used in combination with the graphical representation in order to verify our decoupling.

\bigskip
\noindent\textbf{Construction.} Fix $p \in [0,1]$ and an initial configuration $\eta_0 \in \{0,1\}^\Z$. Sites $x \in \Z$ with $\eta_0(x) = 1$ are said to contain a particle at $t=0$. We label each particle by its starting position.
Each site $x \in \Z$ is given two independent Poisson clocks $\cP^r(x)$, $\cP^l(x)$ with rates $p$ and $1-p$, respectively.
In case $\eta_0(x)=1$, particle $x$ will use $\cP^r(x)$, $\cP^l(x)$ to perform its future jumps.
Specifically, whenever $\cP^r(x)$ (resp.\ $\cP^l(x)$) rings, particle $x$ attempts to jump to the right (resp.\ left).
If the target site is occupied, the jump is suppressed and the particle remains at its current position; otherwise, the particle moves to the target site. 
Note that, with this construction, each fixed particle has rate at most $p$ to jump to the right and at most $1-p$ to jump to the left.

\bigskip
\noindent\textbf{Higher-class particles.} It will be useful for us to divide particles into ordered classes.
Particles in the same class are treated as indistinguishable, but the evolution of each class is influenced by particles of lower class.

Using the previous construction, this can be introduced as follows. For each $k \in \N$,
particles of class $k$ try to jump to the right/left according to their own clocks as usual. 
The jump is suppressed if the target site is occupied by a particle with class at most $k$. 
In case the target site is occupied by a particle with class $k+1$ or higher, the jump is not suppressed, and instead the two sites involved exchange occupations.

In this construction, for each $k$, the collection of particles of class at most $k$ are themselves distributed as an asymmetric exclusion process. 
Thus a $k$th-class particle can be understood as a ``particle that is invisible to all the particles with class at most $k-1$''. Alternatively, a particle is a $k$th-class particle if adding or removing it does not affect the evolution of the particles with class at most $k-1$. 

Finally, note that a higher-class particle may perform more jumps than in the original process,
as it may be forced to jump when a lower-class particle jumps on top of it.
However, this additional jump rate is at most $p$ to the left and $1-p$ to the right, so that particles of any class have jump rates at most $1$ in either direction. In particular, their displacement between two times $t$, $t+s$ in a given direction is stochastically dominated by a Poisson random variable with parameter $s$. This observation will help us control which particles influence each region when verifying the decoupling inequality \eqref{e:DEC}.

\begin{proof}[Proof of Proposition~\ref{p:DEC_ASEP}]
Fix $p, \rho \in [0,1]$. Let $\eta_0(x)$ be i.i.d.\ Bernoulli($\rho$), and define $\eta = (\eta_t)_{t \geq 0}$ using the construction above.
Fix $v_\circ>1$, $\gamma_\circ>1$, $\kappa_{\circ} > 0$ and $C_\circ>0$. 
Fix $\bar{v} \in (1,v_\circ)$ and let  $\varepsilon \in (0,1)$ such that
$(1-7\varepsilon) v_\circ \geq \bar{v}$.
Set
\begin{equation}
\label{e:prdecASEP1}
\usebigconstant{c:decoupling_2} = \frac{2 v_\circ}{\varepsilon}.
\end{equation}
Let $B_1$, $B_2$ be two regions as in the statement of Assumption~\eqref{e:DEC} only satisfying 
\begin{equation}\label{e:prdecASEP2}
\dist_H \geq v_\circ \dist_V + \usebigconstant{c:decoupling_2} s
\end{equation}
at first; the constant $\usebigconstant{c:decoupling_3}$ will be identified later.
By translation invariance, we may assume that $a=b=0$ so $c=\dist_H$ and $d=\dist_V+s$.
Split the interval $[0, \dist_{H})$ into four subintervals: $[0, \Delta)$, $[\Delta, 2\Delta)$, $[2\Delta, 3\Delta)$, and $[3\Delta, \dist_{H})$, where $\Delta = 2\varepsilon\dist_{H}$. 
Declare particles starting in $[3\Delta, +\infty)$ as being first class, particles starting at $(-\infty, \Delta)$ as second class and particles starting in $[\Delta, 3\Delta)$ as third class. Define the events
\begin{equation*}
\begin{split}
G_1 & = \{\text{all first-class particles remain in $(2\Delta, \infty)$ until time $s$}\};\\
G_{2} & = \{\text{all second-class particles remain in $(-\infty, 2\Delta)$ until time $s$}\};\\
G_3 & = \{\text{all third-class particles remain in $(0, \infty)$ until time $s$}\};\\
G_{23} & = \{\text{all second- and third-class particles remain in $(-\infty, \dist_H)$ until time $\dist_V+2s$}\},
\end{split}
\end{equation*}
and denote by $G$ their intersection.
Note that, on $G$, the configuration of $\eta$ inside $B_1$, $B_2$ is a function of $[\eta_0(x), \cP^r(x), \cP^l(x)]_{x < \Delta}$ and $[\eta_0(x), \cP^r(x), \cP^l(x)]_{x \geq 3\Delta}$, respectively, which are independent. 
Thus, for any $f_1, f_2$ as in the statement of Assumption~\eqref{e:DEC},
\begin{equation*}
\EE^{\rho}[f_1f_2]\leq \EE^{\rho}[f_1] \EE^{\rho}[f_2]+3\PP^{\rho}[G^{c}].
\end{equation*}

Now it suffices to bound the probability of $G^{c}$. 
Since particles move to each direction with rate at most one, we can bound,
for a given site $x$,
\begin{equation*}
\PP\left( \eta_0(x) = 1, \text{ particle $x$ travels distance $D$ by time $t$} \right) \leq 2 \rho \PP(N_t \geq D) \leq 2 e^{2t -D},
\end{equation*}
where $N$ is a Poisson process with rate $1$ and we used Lemma~\ref{l:PoiConc}.
Thus
\[
\PP(G_1^c) \leq \sum_{x \geq 3 \Delta} 2 e^{2s} e^{-(x-2\Delta)} \leq 4 e^{2s-\Delta} \leq 4 e^{-\varepsilon \dist_H}
\]
since $\Delta = 2 \varepsilon \dist_H$ and $\varepsilon \dist_H \geq 2 s$ by \eqref{e:prdecASEP1}--\eqref{e:prdecASEP2}.
Analogously, we can bound
\[
\PP(G_2^c) \leq 4 e^{-\varepsilon \dist_H} \quad \text{ and } \quad \PP(G_3^c) \leq 4 e^{-\varepsilon \dist_H}.
\]
For $G_{23}^c$, choose $\theta>0$ small enough such that $e^{\theta}-\bar{v} \theta-1<0$
and estimate
\[
\PP\left(N_t \geq D \right) \leq e^{t(e^\theta-1) - \theta D} = e^{t(e^\theta - \bar{v} \theta - 1) - \theta(D-t\bar{v})} \leq e^{-\theta(D-t \bar{v})}.
\]
Taking $t = \dist_V + 2s$ and noting that
\[
\dist_H - 3\Delta - \varepsilon \dist_H = (1-7\varepsilon)\dist_H \geq (1-7\varepsilon)v_\circ(\dist_V + 2s) \geq \bar{v} (\dist_V + 2s),
\]
we obtain, for some constant $c>0$,
\[
\PP(G_{23}^c) \leq \sum_{k=0}^\infty 2 \rho e^{-\theta(\varepsilon \dist_H + k)} \leq c e^{-\varepsilon \theta \dist_H}.
\]
The previous bounds together yield
\begin{equation}
\EE^{\rho}[f_1f_2]\leq \EE^{\rho}[f_1] \EE^{\rho}[f_2]+ (12+c)e^{-\varepsilon \theta \dist_H}.
\end{equation}
To conclude, choose $\usebigconstant{c:decoupling_3}$ large enough such that $\dist_{H} \geq \usebigconstant{c:decoupling_3}$ implies
\begin{equation*}
(12+c) e^{-\varepsilon \theta \dist_H} \leq C_{\circ} e^{-\kappa_{\circ} (\log \dist_{H})^{\gamma_{\circ}}}.
\qedhere
\end{equation*}
\end{proof}

\appendix

\section{Deviation estimates for submartingales}
\label{s:martingale_appendix}

In this subsection we give the proof of Theorem~\ref{t:largedrift}.
It is based on a deviation estimate for submartingales, cf.\ Proposition~\ref{p:dev_submart} below.
Before we start, we state a simple Chernoff bound for Poisson random variables that is used throughout the paper.
\begin{lemma}\label{l:PoiConc}
Let $N \sim \textnormal{Poisson}(\lambda )$ where $\lambda >0$. For every $u>0$,
\begin{equation*}
\PP\left[N\geq u \right] \leq e^{2\lambda} e^{-u}.
\end{equation*}
\end{lemma}
\begin{proof}
Simply observe that
$\displaystyle
\PP\left[ N\geq u \right] =  \PP\left[e^{N}\geq e ^{u} \right] \leq \EE\left[ e^{N}\right] e^{-u} = e^{(e -1)\lambda - u}
$.
\end{proof}

Our first observation is as follows. Recall Sections~\ref{ss:RW} and \ref{ss:properties}.
\begin{lemma}\label{l:submart}
Suppose that $\inf_{\xi \in E} \{\alpha(\xi) - \beta(\xi)\} > u$ for some $u \in \R$.
Then $M_t := X_t - tu$ is a submartingale under $\PP$ with respect to the filtration $\cF_t = \sigma(\eta, (X_u)_{0 \leq u \leq t})$.
\end{lemma}
\begin{proof}
Fix $u \in \R$ as in the statement.
We will use the notation $\cO(x)$ for a function whose absolute value is bounded by $C|x|$,
where $C$ is a positive constant depending only on $\lambda$ and $|u|$.
For a random variable $Z \geq 0$ and an event $A$, we will write $\EE[Z;A] = \EE[Z \1_A]$.

Let $\delta = \inf_{\xi \in E} \{\alpha(\xi) - \beta(\xi)\} - u >0$.
For $t\geq 0$, an event $A_t \in \cF_t$ and $x \in \Z$,
write $A_t(x) = A_t \cap \{X_t = x\} \in \cF_t$.
For $s \geq 0$, let $\bar{X}_{t,s} = \sup_{0 \leq u \leq s}|X_{t+u}-X_t|$ and note that
\begin{equation}
\label{e:prsubm1}
\begin{split}
& \EE \left[|X_{t+s}-X_t|; A_t(x), \bar{X}_{t,s} \geq 2\right] \leq \EE \left[N_s^{x,t} ; A_t(x), N_s^{x,t} \geq 2 \right] \\
= & \, \EE \left[N_s^{x,t} ; N_s^{x,t} \geq 2 \right] \PP(A_t(x)) =  \cO(s^2) \PP(A_t(x)),
\end{split}
\end{equation}
where $N^y$ is as in Lemma~\ref{l:bdincr} and we used that $E[N^y_s; N^y_s \geq 2] \leq (\lambda s)^2$. 

Recall the definition of the Poisson point processes $\Pi_\alpha, \Pi_\beta, \Pi_\Lambda$ in Sections~\ref{ss:RW} and \ref{ss:properties}.
Define random variables $Z^\alpha_{t,s}(x) = \Pi_\alpha (\{x\} \times (t, t+s])$ and analogously $Z^\beta_{t,s}(x)$, $Z^\Lambda_{t,s}(x)$.
Note that $Z^\alpha_{t,s}(x) + Z^\beta_{t,s}(x) \leq Z^\Lambda_{t,s}(x)$.
Consider the events
\[
\begin{split}
G^+_{t,s}(x) & = \left\{ Z^\alpha_{t,s}(x) = 1, Z^{\beta}_{t,s}(x) = Z^{\Lambda}_{t,s}(x-1) = Z^{\Lambda}_{t,s}(x+1) =0 \right\}, \\
G^-_{t,s}(x) & = \left\{ Z^\beta_{t,s}(x) = 1, Z^{\alpha}_{t,s}(x) = Z^{\Lambda}_{t,s}(x-1) = Z^{\Lambda}_{t,s}(x+1) =0 \right\},\\
G^0_{t,s}(x) & = \left\{ Z^\alpha_{t,s}(x) = Z^{\beta}_{t,s}(x) =0 \right\},
\end{split}
\]
as well as $G_{t,s}(x) =  G^+_{t,s}(x) \cup G^-_{t,s}(x)$ and $\bar{G}_{t,s}(x) = G_{t,s}(x) \cup G^0_{t,s}(x)$.
Note first that 
\[
\begin{split}
\PP(G^+_{t,s}(x)) & = e^{-2\Lambda s} \EE \left[ e^{-\int_t^{t+s} \alpha(\eta_\tau(x)) + \beta(\eta_\tau (x)) \dd \tau} \int_t^{t+s} \alpha(\eta_\tau(x)) \dd \tau  \right], \\
\PP(G^-_{t,s}(x)) & = e^{-2\Lambda s} \EE \left[ e^{-\int_t^{t+s} \alpha(\eta_\tau(x)) + \beta(\eta_\tau (x)) \dd \tau} \int_t^{t+s} \beta(\eta_\tau(x)) \dd \tau  \right].
\end{split}
\]
Moreover,
\begin{equation}
X_{t+s} - X_t = 
\begin{cases}
1  & \text{ on } G^+_{t,s}(x) \cap \{X_t = x\},\\
-1 & \text{ on } G^-_{t,s}(x) \cap \{X_t = x\},
\end{cases}
\end{equation}
so that
\begin{equation}\label{e:prsubm2}
\begin{split}
& \EE \left[ X_{t+s} - X_t ; A_t(x), G_{t,s}(x) \right] 
 =  \left\{ \PP(G^+_{t,s}(x)) - \PP(G^-_{t,s}(x)) \right\} \PP(A_t(x)) \\
\geq  & \, (u+\delta) s e^{-3\Lambda s} \PP(A_t(x)) 
 = \left\{ s(u+\delta) + \cO(s^2) \right\} \PP(A_t(x)).
\end{split}
\end{equation}
Now note that $\bar{G}_{t,s}(x)^c \subset \{Z^\Lambda_{t,s}(x-1) + Z^\Lambda_{t,s}(x)+ Z^\Lambda_{t,s}(x+1) \geq 2 \}$,
and thus
\[
\PP(\bar{G}_{t,s}(x)^c) = \cO(s^2).
\]
Since $G_{t,s}(x) \cap \{X_t = x\} \subset \{\bar{X}_{t,s} = 1\}$ and $G^0_{t,s} \cap \{X_t =x\} \subset  \{\bar{X}_{t,s} = 0\}$,
\begin{equation}
\label{e:prsubm3}
\begin{split}
& \left| \EE \left[X_{t+s}-X_t; A_t(x), \bar{X}_{t,s} = 1\right] - \EE \left[X_{t+s}-X_t; A_t(x), G_{t,s}(x)\right] \right|\\
\leq & \PP\left( A_t(x), \bar{X}_{t,s} = 1, \bar{G}_{t,s}(x)^c \right)
\leq \PP(\bar{G}_{t,s}(x)^c) \PP(A_t(x)) = \cO(s^2) \PP(A_t(x)).
\end{split}
\end{equation}
Putting \eqref{e:prsubm1}--\eqref{e:prsubm3} together, we obtain
\begin{equation*}
\begin{split}
\EE \left[ M_{t+s} - M_t ; A_t \right] 
& = \left\{- s u + \cO(s^2) \right\} \PP(A_t) + \sum_{x \in \Z} \EE \left[ X_{t+s} - X_t ; A_t(x), G_t(x) \right] \\
& = \left\{s \delta + \cO(s^2) \right\} \PP(A_t) \geq \tfrac12 s \delta \PP(A_t)
\end{split}
\end{equation*}
whenever $s \leq s_0$ for some $s_0 = s_0(\lambda, |u|) > 0$.
This implies that, for such $s$ and any $t\geq0$, $E[M_{t+s}|\cF_t] \geq M_t$ almost surely.
To pass this result to any $s>0$, write $t+s = t+s_1 +s_2 + \ldots +s_k$ where each $s_i \leq s_0$, and recursively apply the previous case.
\end{proof}

Motivated by Lemmas~\ref{l:submart} and \ref{l:bdincr}, 
we provide next a deviation estimate for submartingales whose increments have uniform exponential tails.
%
\newconstant{c:hypothesis_increments_1}
\newconstant{c:hypothesis_increments_2}
\newconstant{c:concentration_dyn_Mar_1}
\newconstant{c:concentration_dyn_Mar_2}
%
\begin{proposition}
\label{p:dev_submart}
Let $(M_{t})_{t\geq 0}$ be a c\`adl\`ag submartingale with respect to a filtration $ (\mathcal{F}_{t})_{t \geq 0}$. Assume that $M_{0}=0$ and that there exist constants $\useconstant{c:hypothesis_increments_1},\useconstant{c:hypothesis_increments_2}>0$  such that
\begin{equation}\label{e:cond_dev_submart}
\PP\Big[\sup_{s \in [0,1]}|M_{t+s}-M_{t}|> u \Big]\leq \useconstant{c:hypothesis_increments_1}e^{- \useconstant{c:hypothesis_increments_2} u}
\quad \text{ for all }t\geq 0  \text{ and } u > 0.
\end{equation}
For all $\varepsilon >0$, there exist constants $\useconstant{c:concentration_dyn_Mar_1},\useconstant{c:concentration_dyn_Mar_2}>0$ such that
\begin{equation}\label{e:dev_submart}
\PP \left[ M_{t} \leq -\varepsilon t\right]\leq \useconstant{c:concentration_dyn_Mar_1} e^{-\useconstant{c:concentration_dyn_Mar_2}t^{1/3}} \quad \text{ for all } t \geq 0.
\end{equation}
\end{proposition}

\begin{proof}
Let us start by studying $M_t$ for integer times. For each $i \in \mathbb{N}$, we denote $W_i=M_i-M_{i-1}$. 
Fix $\theta \in (0,\tfrac12)$. For $n \in \mathbb{N}$, define the random variables 
\begin{equation*}
\begin{split}
Y_i & = W_i\mathbbm{1}_{\{|W_i|\leq n^{\theta}\}}- \EE\left[W_i \mathbbm{1}_{\{|W_i|\leq n^{\theta}\}}| \mathcal{F}_{i-1} \right],\\
Z_i & = W_i\mathbbm{1}_{\{|W_i| > n^{\theta}\}}- \EE\left[W_i \mathbbm{1}_{\{|W_i|> n^{\theta}\}}| \mathcal{F}_{i-1} \right].
\end{split}
\end{equation*}
Then $Y_i$ and $Z_i$ are martingale differences with respect to the filtration $ (\mathcal{F}_{i})_{i \geq 0}$, and
\begin{equation}
 \sum_{i=1}^{n} Y_i + \sum_{i=1}^{n} Z_i = M_n - \sum_{i=1}^{n} \EE \left[M_i-M_{i-1} | \mathcal{F}_{i-1} \right] \leq M_{n},
\end{equation}
since $M_{n}$ is a submartingale with $M_{0}=0$. In particular, for $\varepsilon > 0$,
\begin{equation}\label{split_martingale}
\PP \Big[ M_{n} \leq -\frac{\varepsilon}{2} n \Big] \leq \PP \Big[ \sum_{i=1}^{n} Y_{i} \leq -\frac{\varepsilon}{4} n \Big] + \PP\Big[ \sum_{i=1}^{n}Z_{i} \leq -\frac{\varepsilon}{4} n \Big].
\end{equation}
To bound the first probability on the right-hand side above, notice that $ \left\lvert Y_i \right\rvert \leq 2n^{\theta}$, 
and therefore we can apply Azuma's inequality (see \cite[page 146]{lesigne}) to obtain
\begin{equation}\label{1split}
 \PP\Big[ \sum_{i=1}^{n}Y_{i} \leq -\frac{\varepsilon}{4} n\Big] \leq e^{-c\frac{n^2}{n \cdot n^{2\theta}}} = e^{-c n^{1-2\theta}}
\end{equation}
for some constant $c>0$.
To bound the second probability, estimate
\begin{equation}
\begin{split}
\EE\left[|Z_{i}|\right]
& \leq 2\EE \left[|W_{i}|\1_{\{|W_i|>n^{\theta}\}}\right]  \\
& =  2 \left\{ n^\theta \PP \left(|W_i|>n^{\theta} \right) + \int_{n^{\theta}}^\infty \PP \left(|W_i| > u\right) \dd u \right\} \\
& \leq 2 \left\{ n^\theta \useconstant{c:hypothesis_increments_1} e^{- \useconstant{c:hypothesis_increments_2} n^\theta}  +  \int_{n^{\theta}}^\infty \useconstant{c:hypothesis_increments_1} e^{- \useconstant{c:hypothesis_increments_2} u } \dd u \right\}
\leq C e^{-\useconstant{c:hypothesis_increments_2} n^\theta}
\end{split}
\end{equation}
for a positive constant $C$. 
Markov's inequality implies
\begin{equation}\label{2split}
\PP\Big[\Big\lvert \sum_{i=1}^{n}Z_{i}\Big\rvert \geq \frac{\varepsilon}{4}n\Big] 
\leq \frac{4}{\varepsilon n} \mathbb{E}\Big[\Big|\sum_{i=1}^{n}Z_{i}\Big| \Big] 
\leq \frac{4}{\varepsilon} C e^{- \useconstant{c:hypothesis_increments_2}n^{\theta}}.
\end{equation}
Choosing $\theta=1/3$ and combining \eqref{split_martingale}, \eqref{1split}, \eqref{2split}, we obtain 
\begin{equation}\label{Prop_for_int_times}
\PP\left[ M_{n} \leq -\frac{\varepsilon}{2} n\right] \leq  e^{-cn^{1/3}} + \frac{4C}{\varepsilon} e^{-\useconstant{c:hypothesis_increments_2}n^{1/3}} \leq \hat{C} e^{-\hat{c} n^{1/3}}
\end{equation}
for some constants $\hat{C}, \hat{c} >0$.
Now, using \eqref{e:cond_dev_submart} and \eqref{Prop_for_int_times}, we conclude that
\begin{equation*}
\begin{split}
\PP\left[ M_{t} \leq -\varepsilon t\right] & \leq \PP\left[ \lvert M_{t} - M_{\lfloor t \rfloor}\rvert \geq \frac{\varepsilon}{2} \lfloor t \rfloor \right] + \PP\left[ M_{\lfloor t  \rfloor} \leq -\frac{\varepsilon}{2} \lfloor t \rfloor \right]\\
&\leq \useconstant{c:hypothesis_increments_1}e^{- \useconstant{c:hypothesis_increments_2} \frac{\varepsilon}{2}\lfloor t \rfloor} 
+ \hat{C}e^{-\hat{c}\lfloor t \rfloor^{1/3}} \leq \useconstant{c:concentration_dyn_Mar_1} e^{-\useconstant{c:concentration_dyn_Mar_2}t^{1/3}},
\end{split}
\end{equation*}
for suitable positive constants $\useconstant{c:concentration_dyn_Mar_1}, \useconstant{c:concentration_dyn_Mar_2}$ as desired. 
\end{proof}

Using Lemma~\ref{l:submart} and Proposition~\ref{p:dev_submart}, we can finish the proof of Theorem~\ref{t:largedrift}.

\begin{proof}[Proof of Theorem~\ref{t:largedrift}]
Let $\varepsilon = \tfrac13 \{\inf_{\xi} [\alpha(\xi) - \beta(\xi)] - v_\circ \} > 0$ and set $v_\star = v_\circ + \varepsilon$, $u = v_\star +\varepsilon$.
By Lemma~\ref{l:submart}, $M_t = X_t - t u$ is a c\`adl\`ag submartingale, and by Lemma~\ref{l:bdincr}, its increments satisfy \eqref{e:cond_dev_submart}.
Thus $\PP(X_t \leq v_\star t) = \PP(M_t \leq -\varepsilon t) \leq \useconstant{c:concentration_dyn_Mar_1} e^{-\useconstant{c:concentration_dyn_Mar_2}t^{1/3}}$
by Proposition~\ref{p:dev_submart}, implying \eqref{e:BAL} for any $\kappa_\star, \gamma_\star>1$ and an appropriate $C_\star>0$.
\end{proof}


\bibliographystyle{plain}
\bibliography{TeseWeb}

\end{document}